\documentclass[10pt]{amsart}

\usepackage{graphicx}
\graphicspath{{./figs/}}
\usepackage{caption,subcaption}

\usepackage[foot]{amsaddr}

\usepackage{amsmath,amssymb,amsfonts,amsthm,bm}

\usepackage{lineno}

\usepackage{algorithm}
\usepackage{algpseudocode}


\newtheorem{thm}{Theorem}[section]
\newtheorem{theorem}[thm]{Theorem}
\newtheorem{lemma}[thm]{Lemma}

\newtheorem{assumption}[thm]{Assumption}

\newtheorem{proposition}[thm]{Proposition}
\theoremstyle{remark}
\newtheorem{remark}[thm]{Remark}
\numberwithin{equation}{section}

\usepackage[pdftex,dvipsnames]{xcolor}

\newcommand{\norm}[1]{\left\Vert#1\right\Vert}

\makeatletter
\newcommand{\tnorm}{\@ifstar\@tnorms\@tnorm}
\newcommand{\@tnorms}[1]{%
\left|\mkern-2.5mu\left|\mkern-2.5mu\left|
#1
\right|\mkern-2.5mu\right|\mkern-2.5mu\right|
}
\newcommand{\@tnorm}[2][]{%
\mathopen{#1|\mkern-2.5mu#1|\mkern-2.5mu#1|}
#2
\mathclose{#1|\mkern-2.5mu#1|\mkern-2.5mu#1|}
}
\makeatother

\newcommand{\jump}[2]{\lbrack\hspace{-1.5pt}\lbrack {#1} 
\rbrack\hspace{-1.5pt}\rbrack_{\raisebox{-2pt}{\scriptsize$#2$}} }

\newcommand{\oversetc}[2]{%
\mathrel{\vbox{\offinterlineskip\ialign{%
\hfil##\hfil\cr
$\scriptstyle #1$\cr
\noalign{\kern 0.2ex}
$#2$\cr
}}}}

\usepackage{hyperref}
\hypersetup{
	plainpages=false,
    colorlinks,
    linktocpage=true,   
    linkcolor={red!50!black},
    citecolor={blue!50!black},
    urlcolor={blue!80!black}
} 
\usepackage[hyperpageref]{backref}

\title[VEM flux recovery on quadtree]{A simple virtual element-based flux recovery on quadtree}

\author{Shuhao Cao}
\address{Department of Mathematics and Statistics, Washington University in St. Louis, 
   St. Louis, MO 63105, USA}
\email{s.cao@wustl.edu}
\thanks{This work was supported in part by the National Science Foundation
under grants DMS-1913080 and DMS-2136075, and no additional revenues are related to this work.}

\date{}



\begin{document}
    
\begin{abstract}
In this paper, we introduce a simple local flux recovery for $\mathcal{Q}_k$ finite element of a scalar coefficient diffusion equation on quadtree meshes, with no restriction on the irregularities of hanging nodes. The construction requires no specific ad hoc tweaking for hanging nodes on $l$-irregular ($l\geq 2$) meshes thanks to the adoption of virtual element families. The rectangular elements with hanging nodes are treated as polygons as in the flux recovery context. An efficient \emph{a posteriori} error estimator is then constructed based on the recovered flux, and its reliability is proved under common assumptions, both of which are further verified in numerics.
\end{abstract}
\maketitle

\section{Introduction}
In this paper, we consider the following diffusion equation on $\Omega\subset \mathbb{R}^2$,
\begin{equation}
\label{eq:model}
\left\{
\begin{aligned}
-\nabla \cdot (\alpha \nabla u)  &= f, \quad \text{ in } \Omega,
\\
u &= 0, \quad \text{ on } \partial \Omega.
\end{aligned} 
\right.
\end{equation}

To approximate \eqref{eq:model}, taking advantage of the adaptive mesh refinement (AMR) to save valuable computational resources, the adaptive finite element method on quadtree mesh is among the most popular ones in the engineering and scientific computing community 
\cite{Demkowicz;Oden;Rachowicz;Hardy:1989Toward}. 
Compared with simplicial meshes, quadtree meshes provide preferable performance in the aspects of the accuracy and robustness. There are lots of mature software packages (e.g., \cite{mfem,Bangerth;Hartmann;Kanschat:2007deal.II}) on quadtree meshes. To guide the AMR, one possible way is through the \emph{a posteriori} error estimation to construct computable quantities to indicate the location that the mesh needs to be refined/coarsened, thus to balance the spacial distribution of the error which improves the accuracy per computing power. Residual-based and recovery-based error estimators are among the most popular ones used. In terms of accuracy, the recovery-based error estimator shows more appealing attributes 
\cite{Zienkiewicz;Zhu:1992superconvergentpart1,Bank.Xu:2003Asymptotically}. 

More recently, newer developments on flux recovery have been studied by many researchers on constructing a post-processed flux in a structure-preserving approximation space. 
Using \eqref{eq:model} as an example, given that the data $f\in L^2(\Omega)$, the flux $-\alpha \nabla u$ is in $\bm{H}(\mathrm{div}):= \{\bm{v}\in \bm{L}^2(\Omega): \nabla \cdot \bm{v}\in L^2(\Omega)\}$, which has less continuity constraint than the ones in 
\cite{Zienkiewicz;Zhu:1992superconvergentpart1,Bank.Xu:2003Asymptotically} which are vertex-patch based with the recovered flux being $H^1(\Omega)$-conforming. The $\bm{H}(\mathrm{div})$-flux recovery shows more robustness than vertex-patch based ones (e.g., \cite{Cai;Zhang:2009Recovery-based,Cai;Cao:2015recovery-based}).

However, these $\bm{H}(\mathrm{div})$-flux recovery techniques work mainly on conforming meshes. For nonconforming discretizations on nonmatching grids, some simple treatment of hanging nodes exists by recovering the flux on a conforming mother mesh \cite{Ern;Vohralik:2009reconstruction}. To our best knowledge, there is no literature about the local $\bm{H}(\mathrm{div})$-flux recovery on a multilevel irregular quadtree meshes. One major difficulty is that it is impossible to recover a robust computable polynomial flux to satisfy the $\bm{H}(\mathrm{div})$-continuity constraint, that is, the flux is continuous in the normal direction on edges with hanging nodes.

More recently, a new class of methods called the virtual element methods (VEM) were introduced in \cite{Beirao-da-Veiga;Brezzi;Cangiani;Manzini:2013principles,Brezzi;Falk;Marini:2014principles}, which can be viewed as a polytopal generalization of the tensorial/simplicial finite element. Since then, lots of applications of VEM have been studied by many researchers. A usual VEM workflow splits the consistency (approximation) and the stability of the method as well as the finite dimensional approximation space into two parts. It allows flexible constructions of spaces to preserve the structure of the continuous problems such as higher order continuities, exact divergence-free spaces, and many others. The VEM functions are represented by merely the degrees of freedom (DoF) functionals, not the pointwise values. In computation, if an optimal order discontinuous approximation can be computed elementwisely, then adding an appropriate parameter-free stabilization suffices to guarantee the convergence under common assumptions on the geometry of the mesh.

The adoption of the polytopal element brings many distinctive advantages, for example, treating rectangular element with hanging nodes as polygons allows a simple construction of $\bm{H}(\mathrm{div})$-conforming finite dimensional approximation space on meshes with multilevel irregularities.
We shall follow this approach to perform the flux recovery for a conforming $\mathcal{Q}_k$ discretization of problem \eqref{eq:model}. Recently, arbitrary level of irregular quadtree meshes have been studied in \cite{Di-Stolfo;Schroder;Zander;Kollmannsberger:2016treatment,Solin;Cerveny;Dolezel:2008Arbitrary-level,Cerveny;Dobrev;Kolev:2019Nonconforming}. Analyses of the residual-based error estimator on 1-irregular (balanced) quadtree mesh can be found, e.g., in \cite{Carstensen;Hu:2009hanging}. 
In the virtual element context, Zienkiewicz-Zhu (ZZ)-type recovery techniques are studied for linear elasticity in \cite{Chi;Beirao-da-Veiga;Paulino:2019simple}, and for diffusion problems in \cite{Guo;Xie;Zhao:2019Superconvergent}. In \cite{Chi;Beirao-da-Veiga;Paulino:2019simple,Guo;Xie;Zhao:2019Superconvergent}, the recovered flux is in $\bm{H}^1$ and associated with nodal DoFs, thus cannot yield a robust estimate when the diffusion coefficient has a sharp contrast \cite{Cai;Zhang:2009Recovery-based,Cai;Cao:2015recovery-based}. 
The first equilibrated flux recovery in $\bm{H}(\mathrm{div})$ for virtual element methods is studied in \cite{Dassi;Gedicke;Mascotto:2021adaptive}.  While \cite{Dassi;Gedicke;Mascotto:2021adaptive} recovers a flux by solving a mixed problem globally, we opt for a cheap and simple weighted averaging locally. 

The major ingredient in our study is an $\bm{H}(\mathrm{div})$-conforming virtual element space modified from the ones used in \cite{Brezzi;Falk;Marini:2014principles,Beirao-da-Veiga;Brezzi;Marini;Russo:2017Serendipity} (Section \ref{sec:vem}). 
Afterwards, an $\bm{H}(\mathrm{div})$-conforming flux is recovered by a robust weighted averaging of the numerical flux, in which some unique properties of the tensor-product type element $\mathcal{Q}_k$ are exploited (Section \ref{sec:recovery}). 
The \emph{a posteriori} error estimator is constructed based on the projected flux elementwisely. The efficiency of the local error indicator is then proved by bounding it above by the residual-based error indicator (Section \ref{sec:efficiency}). The reliability of the recovery-based error estimator is then shown under certain assumptions (Section \ref{sec:reliability}). These estimates are verified numerically by some common AMR benchmark problems implemented in a publicly available finite element software library $i$FEM \cite{Chen:2008innovative} (Section \ref{sec:numerics}).

\section{Preliminaries}
\subsection{Discretization and notations}
If $\Omega$ is not a rectangle, $u$ is extended by 0 to an $\widetilde{\Omega}$ that is rectangular, therefore without loss of generality, we assume $\Omega$ is partitioned into a shape-regular $\mathcal{T} = \{K\}$ with rectangular elements, and $\alpha := \alpha_K$ is assumed to be a piecewise, positive constant with respect to $\mathcal{T}$.
The weak form of problem \eqref{eq:model} is then discretized in a tensor-product finite element space as follows, 
\begin{equation}
\label{eq:model-fem}
(\alpha \nabla u_{\mathcal{T}}, \nabla v_{\mathcal{T}}) = (f, v_{\mathcal{T}}),\quad \forall v_{\mathcal{T}}\in \mathcal{Q}_k(\mathcal{T})\cap H_0^1(\Omega),
\end{equation}
in which the standard notation is opted. $(\cdot,\cdot)_D$ denotes the inner product on $L^2(D)$, and 
$\norm{\cdot}_D:=\sqrt{(\cdot,\cdot)_D}$, with the subscript omitted when $D= 
\Omega$. The discretization space is
\[
\mathcal{Q}_k(\mathcal{T}) := \{v\in H^1(\Omega): v|_K\in \mathbb{Q}_k(K), \;\forall K\in \mathcal{T} \}.
\]
and on $K = [a,b]\times [c,d]$
\[
\mathbb{Q}_k(K) := \mathbb{P}_{k,k}(K) = \big\{p(x)q(y), \;
p\in \mathbb{P}_k([a,b]), q\in \mathbb{P}_k([c,d]) \big\},
\]
where $\mathbb{P}_k(D)$ stands for the degree no more than $k$ polynomial defined on $D$.
Henceforth, we shall simply denote $\mathcal{Q}_k(\mathcal{T}) =: \mathcal{Q}_k$ when no ambiguity arises.

On $K$, the sets of 4 vertices, as well as 4 edges of the same generation with $K$, are denoted by $\mathcal{N}_K$ and $\mathcal{E}_K$, respectively. The sets of nodes and edges in $\mathcal{T}$ are denoted by 
$\mathcal{N}:= \bigcup_{K\in \mathcal{T}} \mathcal{N}_K$ and $\mathcal{E}:= \bigcup_{K\in \mathcal{T}} \mathcal{E}_K$. 
A node $\bm{z} \in  \mathcal{N}$ is called a hanging node if it is on $\partial K$ but is not counted as a vertex of $K\in \mathcal{T}$, and we denote the set of hanging nodes as $\mathcal{N}_H$
\begin{equation}
\label{eq:node-hanging}
\mathcal{N}_H := \{\bm{z}\in \mathcal{N}: \exists K\in \mathcal{T}, \bm{z} \in \partial K \backslash \mathcal{N}_K \}
\end{equation}
Otherwise the node $\bm{z} \in  \mathcal{N}$ is a regular node. If an edge  $e\in \mathcal{E}$ contains at most $l$ hanging nodes, the partition $\mathcal{T}$, as well as the element these hanging nodes lie on, is called $l$-irregular. 

For each edge $e\in \mathcal{E}$, a unit normal vector $\bm{n}_e$ is fixed by specifying its direction pointing rightward for vertical edges, and upward for horizontal edges. If an exterior normal of an element on this edge shares the same orientation with $\bm{n}_e$, then this element is denoted by $K_-$, otherwise it is denoted by $K_+$, i.e., $\bm{n}_e$ is pointing from $K_-$ to $K_+$. The intersection of the closures of $K_+, K_-$ is always an edge $e\in \mathcal{E}$. 
However, we note that by the definition in \eqref{eq:node-hanging} it is possible that $e\in\mathcal{E}_{K_+}$ but not in $\mathcal{E}_{K_-}$ or vice versa, if there exists a hanging node on $e$ (see e.g., Figure \ref{fig:node-hanging}).
For any function or distribution $v$ well-defined on the two elements, define $\jump{v}{e} = v^- - v^+$ on an edge 
$e\not\subset\partial \Omega$, in which $v^-$ and $v^+$ are defined in the limiting sense $v^{\pm} = \lim_{\epsilon\to 
0^{\pm}}v(\bm{x}+\epsilon \bm{n}_e)$ for $\bm{x}\in e$. If $e$ is a boundary edge, 
the function $v$ is extended by zero outside the domain to compute $\jump{v}{e}$. 
Furthermore, the following notation denotes a weighted average of $v$ on edge $e$ for a weight $\gamma\in [0,1]$,
\[
    \{v\}^{\gamma}_e := \gamma v^- + (1-\gamma) v^+.
\]
\subsection{Virtual element spaces}
\label{sec:vem}
\begin{figure}[h]
\centering
\includegraphics[width=2in]{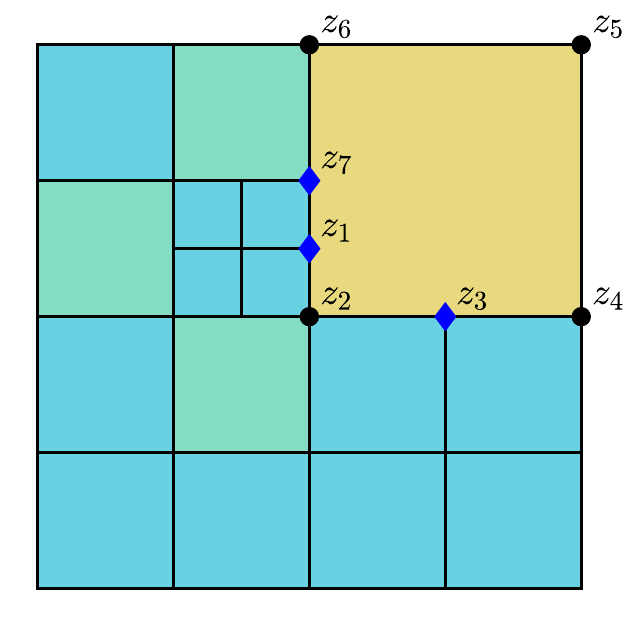}
\caption{For the upper right element $K\in \mathcal{T}$, $\mathcal{N}_K =\{z_2, z_4, z_5, z_6\}$. For $K\in \mathcal{T}_{\mathrm{poly}}$, $\mathcal{N}_K =\{z_i\}_{i=1}^7$ }
\label{fig:node-hanging}
\end{figure}

In this subsection, the quadtree mesh $\mathcal{T}$ of interest is embedded into a polygonal mesh 
$\mathcal{T} \hookrightarrow \mathcal{T}_{\mathrm{poly}}=\{K_{\mathrm{poly}}\}$. On any given quadrilateral element $K$, for example we consider a $v_{\mathcal{T}} \in \mathbb{Q}_1(K)$, it has 4 degrees of freedom associated with 4 nodes $\{z\}$. Its numerical flux $-\alpha \nabla v_{\mathcal{T}} \cdot \bm{n}$ is well-defined on the 4 edges $\{e\}$ locally on $K$, such that on each edge it is a polynomial defined on the whole edge, regardless of the number of hanging nodes on that edge. 
Using Figure \ref{fig:node-hanging} as an example, on the upper right element $K$,
$\nabla v_{\mathcal{T}}|_{K}\cdot \bm{n}|_{\scriptsize{\oversetc{\rightharpoonup}{z_2 z_6}}}\in \mathbb{P}_1(\oversetc{\rightharpoonup}{z_2 z_6})$ is a linear function in $y$-variable. 

For the embedded element $K_{\mathrm{poly}} \in \mathcal{T}_{\mathrm{poly}}$, which geometrically coincides with $K$, it includes all the hanging nodes, while the set of edges are formed accordingly as the edges of the cyclic graph of the vertices. We shall denote the set of all edges on $\mathcal{T}_{\mathrm{poly}}$ as $\mathcal{E}_{\mathrm{poly}}$. Using Figure \ref{fig:node-hanging} as example, it is possible to define a flux on $K$ with piecewise linear normal component on $\oversetc{\rightharpoonup}{z_2 z_6}$ which now consists of three edges on $\partial K_{\mathrm{poly}}$.

Subsequently, $K_{\mathrm{poly}}\in \mathcal{T}_{\mathrm{poly}}$ shall be denoted by simply 
$K\in \mathcal{T}_{\mathrm{poly}}$ in the context of flux recovery, and the notion $e\subset \partial K$ denotes an edge on the boundary of $K$, which takes into account of the edges formed with one end point or both end points as the hanging nodes.

On $\mathcal{T}_{\mathrm{poly}}$, we consider the following Brezzi-Douglas-Marini-type virtual element modification inspired by the ones used in  
\cite{Brezzi;Falk;Marini:2014principles,Beirao-da-Veiga;Brezzi;Marini;Russo:2017Serendipity}.
The local space on a $K\in \mathcal{T}_{\mathrm{poly}}$ is defined as for $k\geq 1$
\begin{equation}
\label{eq:space-vhdiv}
\begin{aligned}
\mathcal{V}_{k}(K) :=  \Big\{ & \bm{\tau}\in \bm{H}(\mathrm{div};K)\cap 
\bm{H}(\mathbf{rot};K):  \\
& \nabla\cdot \bm{\tau}\in \mathbb{P}_{k-1}(K), \quad \nabla \times \bm{\tau} = 0,
\\
& \bm{\tau}\cdot\bm{n}_e \in \mathbb{P}_{k}(e), \;\forall e\subset \partial K
\Big\}.
\end{aligned} 
\end{equation}
An $\bm{H}(\mathrm{div})$-conforming global space for recovering the flux is then 
\begin{equation}
\mathcal{V}_k := \bigl\{\bm{\tau}\in \bm{H}(\mathrm{div}): \bm{\tau}|_K \in \mathcal{V}_{k}(K), 
\;\text{ on } K\in \mathcal{T}_{\mathrm{poly}} \bigr\}.
\end{equation}
Next we turn to define the degrees of freedom (DoFs) of this space. To this end, we define the set of scaled monomials $\mathbb{P}_{k}(e)$ on an edge $e$. $e$ is parametrized by $[0,h_e]\ni s\mapsto \bm{a} + s\bm{t}_e$, where $\bm{a}$ is the starting point of $e$, and $\bm{t}_e$ is the unit tangential vector of $e$. The basis set for $\mathbb{P}_{k}(e)$ is chosen as: 
\begin{equation}
\label{eq:monomials-edge}
\mathbb{P}_{k}(e):=\operatorname{span}\left\{1, \frac{s-m_{e}}{h_{e}},\left(\frac{s-m_{e}}{h_{e}}\right)^{2}, \ldots,\left(\frac{s-m_{e}}{h_{e}}\right)^{k}\right\},
\end{equation}
where $m_e = h_e/2$ representing the midpoint when using this parametrization. Similar to the edge case, $\mathbb{P}_{k}({K})$'s basis set is chosen as follows (see e.g., \cite{Beirao-da-Veiga;Brezzi;Cangiani;Manzini:2013principles}): 
\begin{equation}
\label{eq:monomials-elem}
\mathbb{P}_{k}({K}):=\operatorname{span}\left\{m_{\alpha}(\bm{x}):=\left(\frac{\bm{x}-\bm{x}_{K}}{h_{K}}\right)^{\bm{\alpha}}, \quad|\bm{\alpha}| \leq k\right\}.
\end{equation}
The degrees of freedom (DoFs) are then set as follows for a $\bm{\tau}\in \mathcal{V}_{k}$:
\begin{equation}
\label{eq:space-dof}
\begin{aligned}
(\mathfrak{e})\; k\geq 1  & \quad  \int_e (\bm{\tau}\cdot \bm{n}_e) m \,\mathrm{d} s, \quad  
\forall m \in \mathbb{P}_{k}(e), & \text{on } \; e\subset \mathcal{E}_{\mathrm{poly}}.
\\
(\mathfrak{i})\; k\geq 2 & \quad  \int_K \bm{\tau}\cdot \nabla m\, \mathrm{d} \bm{x} , \quad  
\forall m\in \mathbb{P}_{k-1}({K})/\mathbb{R}   & \text{on } \; K\in \mathcal{T}_{\mathrm{poly}}.
\end{aligned}
\end{equation}

\begin{remark}
We note that in our construction, the degrees of freedom to determine the curl of a VEM function originally in \cite{Brezzi;Falk;Marini:2014principles} are replaced by a curl-free constraint thanks to the flexibility to virtual element. The reason why we opt for this subspace is that the true flux $-\alpha \nabla u$ is locally curl-free since we have assumed that $\alpha$ is a piecewise constant.
The unisolvency of the set of DoFs \eqref{eq:space-dof} including the curl-part can be found in \cite{Brezzi;Falk;Marini:2014principles}. While for the modified space \eqref{eq:space-vhdiv}, a simplified argument is in the proof of Lemma \ref{lem:normeq-V}. 
\end{remark}

\section{Flux recovery}
\label{sec:recovery}

As the data $f\in L^2(\Omega)$, the true flux $\bm{\sigma} = -\alpha\nabla u\in \bm{H}(\mathrm{div})$. Consequently, we shall seek a postprocessed flux $\bm{\sigma}_{\mathcal{T}}$ in $\mathcal{V}_{k}\subset \bm{H}(\mathrm{div})$ by specifying the DoFs in \eqref{eq:space-dof}. Throughout this section, whenever considering an element $K\in \mathcal{T}$, we treat it a polygon as 
$K\in \mathcal{T}_{\mathrm{poly}}$.

\subsection{Virtual element-based flux recovery}
Consider $-\alpha_K \nabla u_{\mathcal{T}}$ which is the numerical flux on $K$. We note that 
$-\alpha_K \nabla u_{\mathcal{T}}|_K \in \mathbb{P}_{k-1,k}(K) \times\mathbb{P}_{k,k-1}(K)$. The normal flux on each edge $e\in \mathcal{E}_{\mathrm{poly}}$ is in $\mathbb{P}_{k}(e)$ as $n_e = (\pm 1, 0)$ and $x=\mathrm{const}$ on vertical edges, $n_e = (0, \pm 1)$ and $y=\mathrm{const}$ on horizontal edges. Therefore, the edge-based DoFs can be computed by a simple averaging thanks to the matching polynomial degrees of the numerical flux to the functions in $\mathcal{V}_k$. 

On each $e = \partial K_+ \cap \partial K_-$, define
\begin{equation}
\label{eq:sigma-w}
\left\{-\alpha \nabla u_{\mathcal{T}} \right\}^{\gamma_e}_e \cdot\bm{n}_e
:= \Big(\gamma_e \left( -\alpha_{K_-} \nabla u_{\mathcal{T}}|_{K_-} \right)
+ (1-\gamma_e) \left( -\alpha_{K_+} \nabla u_{\mathcal{T}}|_{K_+} \right)\Big)\cdot \bm{n}_e,
\end{equation}
where 
\begin{equation}
    \gamma_e := \frac{\alpha_{K_+}^{1/2}}{\alpha_{K_+}^{1/2} + \alpha_{K_-}^{1/2}}.
\end{equation}
First for both $k=1$ and $k\geq 2$ cases, 
we set the normal component of the recovered flux is set as 
\begin{equation}
\label{eq:sigma-n}
 \bm{\sigma}_{\mathcal{T}}\cdot \bm{n}_e = \left\{-\alpha \nabla u_{\mathcal{T}} \right\}^{\gamma_e}_e \cdot\bm{n}_e.   
\end{equation}
In the lowest order case $k= 1$, $\nabla \cdot \bm{\sigma}_{\mathcal{T}}$ is a constant on $K$ by \eqref{eq:space-vhdiv}, thus the construction \eqref{eq:sigma-n} alone, which consists the edge DoFs $(\mathfrak{e})$ in \eqref{eq:space-dof},  can determine the divergence $\nabla \cdot \bm{\sigma}_{\mathcal{T}}$ in $K$ as follows
\begin{equation}
\label{eq:sigma-div1}
|K|\nabla\cdot \bm{\sigma}_{\mathcal{T}} = \int_K \nabla\cdot\bm{\sigma}_{\mathcal{T}} \mathrm{d} \bm{x} = 
\int_{\partial K} \bm{\sigma}_{\mathcal{T}} \cdot\bm{n}_{\partial K}\mathrm{d} s
= \sum_{e\subset \partial K}  \int_e \bm{\sigma}_{\mathcal{T}} \cdot\bm{n}_{\partial K}|_e \mathrm{d} s.   
\end{equation}

If $k\geq 2$, after the normal component \eqref{eq:sigma-n} is set, furthermore on each $K$, denote $\Pi_{k-1}$ stands for the $L^2$-projection to $\mathbb{P}_{k-1}(K)$, and we let
\begin{equation}
\label{eq:sigma-div}
\nabla \cdot \bm{\sigma}_{\mathcal{T}} = \Pi_{k-1} f + c_K.
\end{equation}
The reason to add $c_K$ is that we have set the normal components of the recovered flux first without relying on the divergence information. While in general $\nabla \cdot \bm{\sigma}_{\mathcal{T}} \neq \Pi_{k-1} f$ as otherwise the divergence theorem will be rendered invalid in \eqref{eq:sigma-div1}. 
As a result, an element-wise constant $c_K$ is added to ensure the compatibility of $\bm{\sigma}_{\mathcal{T}}$ locally on each $K$. It is straightforward to verify that $c_K$ has the following form, and later we shall show that $c_K$ does not affect the efficiency as well as the reliability of the error estimates.
\begin{equation}
\label{eq:sigma-c}
 c_K = \frac{1}{ |K| }\left(-\int_K \Pi_{k-1} f \mathrm{d} \bm{x} 
+ \sum_{e\subset\partial K} \int_e \left\{-\alpha \nabla u_{\mathcal{T}} \right\}^{\gamma_e}_e 
\cdot\bm{n}_{\partial K}|_e  \mathrm{d} s\right),   
\end{equation}
Consequently for $k\geq 2$, the set $(\mathfrak{i})$ of DoFs can be set as: $\forall q\in \mathbb{P}_{k-1}(K)$
\begin{equation}
\label{eq:sigma-i}
\bigl(\bm{\sigma}_{\mathcal{T}},\nabla q\bigr)_K
= -\left(\Pi_{k-1} f + c_K, q\right)_K + \sum_{e\subset \partial K}
\left(\left\{-\alpha \nabla u_{\mathcal{T}} \right\}^{\gamma_e}_e \cdot\bm{n}_{\partial K}|_e,q\right)_{e}.
\end{equation}

\subsection{Locally projected flux}
To the end of constructing a computable local error indicator, inspired by the VEM formulation \cite{Brezzi;Falk;Marini:2014principles}, the recovered flux is projected to a space with a much simpler structure. A local oblique projection ${\Pi}: \bm{L}^2(K) \to \nabla \mathbb{P}_k(K), \; \bm{\tau}\mapsto {\Pi} \bm{\tau}$ is defined as follows:
\begin{equation}
\label{eq:sigma-proj}
\bigl({\Pi} \bm{\tau},\nabla p\bigr)_K = \bigl(\bm{\tau},\nabla p\bigr)_K,\quad \forall p\in \mathbb{P}_k(K)/\mathbb{R}.
\end{equation}
Next we are gonna show that this projection operator can be straightforward computed for vector fields in $\mathcal{V}_k(K)$.
\subsubsection{$k=1$}
When $k = 1$, we can compute the right hand side of \eqref{eq:sigma-proj} as follows:
\begin{equation}
\label{eq:sigma-proj-linear}
\bigl(\bm{\tau},\nabla p\bigr)_K
= -\bigl(\nabla \cdot \bm{\tau}, p\bigr)_K + \bigl(\bm{\tau}\cdot\bm{n},p\bigr)_{\partial K}.
\end{equation}
By definition of the space \eqref{eq:space-vhdiv} when $k=1$, $\nabla \cdot \bm{\tau}$ is a constant on $K$ and can be determined by edge DoFs $(\mathfrak{e})$ in \eqref{eq:space-dof} similar to \eqref{eq:sigma-div1}.
Moreover, $p|_e\in \mathbb{P}_1(e)$, thus the boundary term can be evaluated using 
DoFs $(\mathfrak{e})$ in \eqref{eq:space-dof}.

\subsubsection{$k\geq 2$}
When $k\geq 2$, the right hand side of \eqref{eq:sigma-proj} can be evaluated following a similar procedure as \eqref{eq:sigma-proj-linear}, if we exploit the fact that $\nabla \cdot \bm{\tau}\in \mathbb{P}_{k-1}(K)$, 
we have
\begin{equation}
\label{eq:sigma-proj-2}
\begin{aligned}
\bigl(\bm{\tau},\nabla p\bigr)_K
& = -\bigl(\nabla \cdot \bm{\tau}, \Pi_{k-1} p\bigr)_K + \bigl(\bm{\tau}\cdot\bm{n},p\bigr)_{\partial K}
\\
& = \bigl(\bm{\tau}, \nabla \Pi_{k-1} p\bigr)_K + \bigl(\bm{\tau}\cdot\bm{n},p - \Pi_{k-1} p\bigr)_{\partial K},
\end{aligned}
\end{equation}
which can be evaluated using both DoF sets $(\mathfrak{e})$ and $(\mathfrak{i})$.

\section{A posteriori error estimation}

Given the recovered flux \(\bm{\sigma}_{\mathcal{T}}\) in Section \ref{sec:recovery}, the recovery-based local error indicator $\eta_{\mathrm{flux},K}$ and the element residual $\eta_{\mathrm{res},K}$ as follows:
\begin{equation}
\begin{gathered}  
\eta_{\mathrm{flux},K} := \big\Vert{\alpha^{-1/2}(\bm{\sigma}_{\mathcal{T}} + \alpha \nabla u_{\mathcal{T}})} \big\Vert_K, 
\\
\text{ and } \;
\eta_{\mathrm{res},K} := \big\Vert{\alpha^{-1/2}(f - \nabla\cdot\bm{\sigma}_{\mathcal{T}}) } \big\Vert_K,
\end{gathered}
\end{equation}
then 
\begin{equation}
\eta_K = 
\left\{
\begin{array}{lc}
\eta_{\mathrm{flux},K} & \text{when }  k = 1,
\\
\left(\eta_{\mathrm{flux},K}^2 + \eta_{\mathrm{res},K}^2 \right)^{1/2}
& \text{when } k\geq 2. 
\end{array}
\right.
\end{equation}

A computable $\widehat{\eta}_{\mathrm{flux},K}$ is defined as:
\begin{equation}
\label{eq:eta-flux}
\widehat{\eta}_{\mathrm{flux},K}:=\big\Vert{\alpha_K^{-1/2}{\Pi}(\bm{\sigma}_{\mathcal{T}} + \alpha_K \nabla u_{\mathcal{T}})} \big\Vert_K,
\end{equation}
with the oblique projection ${\Pi}$ defined in \eqref{eq:sigma-proj}. The stabilization part $\widehat{\eta}_{\mathrm{stab},K}$ is
\begin{equation}
\label{eq:eta-stab}
\widehat{\eta}_{\mathrm{stab},K}:=\big\vert{\alpha_K^{-1/2}(\operatorname{I}-{\Pi})(\bm{\sigma}_{\mathcal{T}} + \alpha_K \nabla u_{\mathcal{T}})} \big\vert_{S,K}.
\end{equation}
Here $|{\cdot}|_{S,K}:= \big(S_K(\cdot, \cdot)\big)^{1/2}$ is seminorm induced by the following stabilization
\begin{equation}
\label{eq:norm-stab}
S_K(\bm{v}, \bm{w}):=\sum_{e\subset \partial K} 
 h_e\big( \bm{v}\cdot\bm{n}_e,  \bm{w}\cdot \bm{n}_e \big)_e
+ \sum_{\alpha\in \Lambda} (\bm{v},\nabla m_{\alpha})_K (\bm{w},\nabla m_{\alpha})_K,
\end{equation}
where $\Lambda$ is the index set for the monomial basis of $\mathbb{P}_{k-1}(K)/\mathbb{R}$ with cardinality $k(k+1)/2 - 1$, i.e., the second term in \eqref{eq:norm-stab} is dropped in the $k=1$ case. We note that this is a slightly modified version of the standard stabilization for an $\bm{H}(\mathrm{div})$-function in \cite{Brezzi;Falk;Marini:2014principles} as we have replaced the edge DoFs by an integral. In Section \ref{sec:appendix-norm} it is shown that the integral-based stabilization still yields the crucial norm equivalence result.

The computable error estimator $\widehat{\eta}$ is then 
\begin{equation}
\label{eq:eta-computable}
\widehat{\eta}^2 = \begin{cases} 
\sum_{K\in \mathcal{T}} \left(\widehat{\eta}_{\mathrm{flux},K}^2 + \widehat{\eta}_{\mathrm{stab},K}^2 \right) =: \sum_{K\in \mathcal{T}} \widehat{\eta}_{K}^2
& \text{when }  k = 1,
\\[5pt]
\sum_{K\in \mathcal{T}} \left(\widehat{\eta}_{\mathrm{flux},K}^2 + \widehat{\eta}_{\mathrm{stab},K}^2 
+ \eta_{\mathrm{res},K}^2\right)=: \sum_{K\in \mathcal{T}} \widehat{\eta}_{K}^2 & \text{when } k\geq 2.
\end{cases}
\end{equation}

\subsection{Efficiency}
\label{sec:efficiency}
In this section, we shall prove the proposed recovery-based estimator $\widehat{\eta}_{K}$ is efficient by bounding it above by the residual-based error estimator. In the process of adaptive mesh refinement, only the computable $\widehat{\eta}_{K}$ is used as the local error indicator to guide a marking strategy of choice. 
\begin{theorem}
\label{thm:efficiency}
Let \(u_{\mathcal{T}}\) be the solution to problem \eqref{eq:model-fem}, and $\widehat{\eta}_{\mathrm{flux},K}$ be the error indicator in \eqref{eq:eta-computable}. On $K\in \mathcal{T}_{\mathrm{poly}}$, 
$\widehat{\eta}_{\mathrm{flux},K}$ can be locally bounded by the residual-based ones:
\begin{equation}
\label{eq:efficiency}
\widehat{\eta}_{\mathrm{flux},K}^2 \lesssim \mathrm{osc}(f; K)^2  + \eta_{\mathrm{elem},K}^2  +  \eta_{\mathrm{edge},K}^2 ,
\end{equation}
where
\begin{align*}
\mathrm{osc}(f;K) &= \alpha_K^{-1/2}h_K \big\Vert f- \Pi_{k-1} f  \big\Vert_K,
\\
\eta_{\mathrm{elem},K} &:=  \alpha_K^{-1/2}h_K \big\Vert f + \nabla\cdot(\alpha \nabla u_{\mathcal{T}}) \big\Vert_K,
\\
\text{and }\; \eta_{\mathrm{edge},K} &:=  \left(\sum_{e\subset \partial K} \frac{h_e}{\alpha_K + \alpha_{K_e}} 
\big\Vert \jump{\alpha \nabla u_{\mathcal{T}}\cdot\bm{n}_e}{} \big\Vert_e^2\right)^{1/2}.
\end{align*}
In the edge jump term, $K_e$ is the element on the opposite side of $K$ with respect to an edge $e\subset \partial K$. The constant depends on $k$ and the number of edges on $\partial K$.
\end{theorem}

\begin{proof}
Let $\alpha^{-1}_K{\Pi}(\bm{\sigma}_{\mathcal{T}} + \alpha_K \nabla u_{\mathcal{T}})=:\nabla p$ on $K$, then 
$p\in \mathbb{P}_k(K)/\mathbb{R}$ and we have
\begin{equation}
\label{eq:eta-f}
\begin{aligned}
\widehat{\eta}_{\mathrm{flux},K}^2 & = \bigl({\Pi}(\bm{\sigma}_{\mathcal{T}} + \alpha_K \nabla u_{\mathcal{T}}), \nabla p \bigr)_K
 = \bigl(\bm{\sigma}_{\mathcal{T}} + \alpha_K \nabla u_{\mathcal{T}}, \nabla p \bigr)_K
\\
& = -\bigl(\nabla \cdot (\bm{\sigma}_{\mathcal{T}} +\alpha_K \nabla u_{\mathcal{T}}), p \bigr)_K
 + \sum_{e\subset \partial K}\int_e \big( \bm{\sigma}_{\mathcal{T}} + \alpha_K \nabla u_{\mathcal{T}}\big)\cdot \bm{n}_{\partial K}\big|_{e} \, p \, \mathrm{d} s.
\end{aligned} 
\end{equation}
By \eqref{eq:sigma-n}, without loss of generality we assume $K=K_-$ (the local orientation of $e$ agrees with the global one, i.e., $\bm{n}_{\partial K}\big|_{e} = \bm{n}_e$), and $K_e = K_+$ which is the element opposite to $K$ with respect to $e$, and $\gamma_e := {\alpha_{K_e}^{1/2}}/({\alpha_{K_e}^{1/2} + \alpha_{K}^{1/2}})$, we have on edge $e\subset\partial K$ 
\begin{equation}
\begin{aligned}
\big( \bm{\sigma}_{\mathcal{T}} + \alpha_K \nabla u_{\mathcal{T}}\big)\cdot \bm{n}_e &= 
\Big( (1-\gamma_e) \alpha_{K} \nabla u_{\mathcal{T}}|_K 
- (1-\gamma_e)\alpha_{K_e} \nabla u_{\mathcal{T}}|_{K_e} \Big)\cdot \bm{n}_e
\\
&= \frac{\alpha_{K}^{1/2}}{\alpha_{K}^{1/2} + \alpha_{K_e}^{1/2}}\jump{\alpha \nabla u_{\mathcal{T}}\cdot \bm{n}_e}{e}.
\end{aligned}
\end{equation}
The boundary term in \eqref{eq:eta-f} can be then rewritten as
\begin{equation}
\label{eq:eta-bd}
\begin{aligned}  
& \int_e \big( \bm{\sigma}_{\mathcal{T}} + \alpha_K \nabla u_{\mathcal{T}}\big)\cdot \bm{n}_e \,p\, \mathrm{d} s
\\
= &\;\int_e \frac{1}{\alpha_{K}^{1/2} + \alpha_{K_e}^{1/2}}\jump{\alpha \nabla u_{\mathcal{T}}\cdot \bm{n}_e}{e} 
\,\alpha_{K}^{1/2}p\, \mathrm{d} s
\\
\lesssim & \; \frac{1}{(\alpha_{K}+ \alpha_{K_e})^{1/2}} 
h_e^{1/2} \big\Vert \jump{\alpha \nabla u_{\mathcal{T}}\cdot\bm{n}_e}{} \big\Vert_e 
\alpha_{K}^{1/2} h_e^{-1/2} \norm{p}_e.
\end{aligned}    
\end{equation}
By a trace inequality on an edge of a polygon (Lemma \ref{lem:trace}), and the Poincar\'{e} inequality for $p\in \mathbb{P}_k(K)/\mathbb{R}$, we have,
\[
 h_e^{-1/2}\|p\|_e \lesssim  h_K^{-1} \|p\|_K + \|\nabla p\|_K \lesssim \|\nabla p\|_K.
\] 
As a result, 
\[
\sum_{e\subset \partial K}\int_e \big( \bm{\sigma}_{\mathcal{T}} + \alpha_K \nabla u_{\mathcal{T}}\big)\cdot \bm{n}_e \,p\, \mathrm{d} s
\lesssim \eta_{\mathrm{edge},K}\, \alpha_{K}^{1/2} \norm{\nabla p}_e
= \eta_{\mathrm{edge},K} \,\widehat{\eta}_{\mathrm{flux},K}.
\]
For the bulk term on $K$'s in \eqref{eq:eta-f}, when $k=1$, by \eqref{eq:sigma-div1}, the representation in \eqref{eq:eta-bd}, and the Poincar\'{e} inequality for $p\in \mathbb{P}_k(K)/\mathbb{R}$ again with $h_K\simeq |K|^{1/2}$, we have
\[
\begin{aligned}
&-\bigl(\nabla \cdot (\bm{\sigma}_{\mathcal{T}} +\alpha_K \nabla u_{\mathcal{T}}), p \bigr)_K \leq 
\left|\nabla \cdot (\bm{\sigma}_{\mathcal{T}} +\alpha_K \nabla u_{\mathcal{T}}) \right| |K|^{1/2} \norm{p}_K   
\\
\leq & \; \frac{1}{ |K|^{1/2}}\left|\int_K \nabla \cdot (\bm{\sigma}_{\mathcal{T}} +\alpha_K \nabla u_{\mathcal{T}}) \, \mathrm{d} \bm{x}\right|
\norm{p}_K   
\\
= & \; \frac{1}{ |K|^{1/2}} \left|\sum_{e\subset \partial K}
\int_{e} (\bm{\sigma}_{\mathcal{T}} +\alpha_K \nabla u_{\mathcal{T}})\cdot \bm{n}_e \, \mathrm{d} s\right| \norm{p}_K   
\\
\leq & \; \left(\sum_{e\subset \partial K} \frac{1}{\alpha_{K}^{1/2} + \alpha_{K_e}^{1/2}} \norm{\jump{\alpha \nabla u_{\mathcal{T}}\cdot \bm{n}_e}{}}_e 
\,\alpha_{K}^{1/2}h_e \right) \norm{\nabla p}
\\
\lesssim &\; \eta_{\mathrm{edge},K} \, \widehat{\eta}_{\mathrm{flux},K}.
\end{aligned}
\]
When $k\geq 2$, by \eqref{eq:sigma-div}, 
\begin{equation}
\label{eq:eta-div}
 \begin{aligned}
& -\bigl(\nabla \cdot (\bm{\sigma}_{\mathcal{T}} +\alpha_K \nabla u_{\mathcal{T}}), p \bigr)_K  = 
-\bigl(\Pi_{k-1} f + c_K + \nabla\cdot(\alpha_K \nabla u_{\mathcal{T}}), p \bigr)_K    
\\
\leq & \; \left( \big\Vert f- \Pi_{k-1} f  \big\Vert_K 
+ \big\Vert f + \nabla\cdot(\alpha \nabla u_{\mathcal{T}}) \big\Vert_K 
+ |c_K| |K|^{1/2}\right)\norm{p}_K.
\end{aligned}   
\end{equation}
The first two terms can be handled by combining the weights $\alpha^{-1/2}$ and $h_K$ from $\norm{p}_K\leq h_K \norm{\nabla p}_K$. For $c_K$, it can be estimated straightforwardly as follows
\begin{equation}
  \begin{aligned}
 c_K |K|^{1/2} &= \frac{1}{ |K|^{1/2} }\Big(-\int_K (\Pi_{k-1} f -f) \mathrm{d} \bm{x} 
- \int_K \big(f + \nabla\cdot (\alpha \nabla u_{\mathcal{T}})\big) \mathrm{d} \bm{x} 
\\
&\quad + \int_K \nabla\cdot (\alpha \nabla u_{\mathcal{T}}) \mathrm{d} \bm{x} 
+ \sum_{e\subset\partial K} \int_e \left\{-\alpha \nabla u_{\mathcal{T}} \right\}^{\gamma_e}_e \cdot\bm{n}_e  \mathrm{d} s\Big)
\\
&\leq   \big\Vert f- \Pi_{k-1} f  \big\Vert_K 
+ \big\Vert f + \nabla\cdot(\alpha \nabla u_{\mathcal{T}}) \big\Vert_K 
\\
& \quad + \frac{1}{ |K|^{1/2}}\sum_{e\subset\partial K} \int_e (\alpha_K \nabla u_{\mathcal{T}} - \left\{\alpha \nabla u_{\mathcal{T}} \right\}^{\gamma_e}_e) \cdot\bm{n}_e  \mathrm{d} s
\\
& \leq  \big\Vert f- \Pi_{k-1} f  \big\Vert_K 
+ \big\Vert f + \nabla\cdot(\alpha \nabla u_{\mathcal{T}}) \big\Vert_K 
\\
& \quad + \sum_{e\subset\partial K} \frac{\alpha_{K}^{1/2}}{\alpha_{K}^{1/2} + \alpha_{K_e}^{1/2}} \norm{\jump{\alpha \nabla u_{\mathcal{T}}\cdot \bm{n}_e}{}}_e .
\end{aligned}
\end{equation}
The two terms on $K$ can be treated the same way with the first two terms in \eqref{eq:eta-div} while the edge terms are handled similarly as in the $k=1$ case. As a result, we have shown 
\[
-\bigl(\nabla \cdot (\bm{\sigma}_{\mathcal{T}} +\alpha_K \nabla u_{\mathcal{T}}), p \bigr)_K
\lesssim \Big( \mathrm{osc}(f; K)  + \eta_{\mathrm{elem},K}  +  \eta_{\mathrm{edge},K} \Big) 
\alpha_K^{1/2}\norm{\nabla p}
\]
and the theorem follows.
\end{proof}

\begin{theorem}
\label{thm:efficiency-stab}
Under the same setting with Theorem \ref{thm:efficiency}, let $\widehat{\eta}_{\mathrm{stab},K}$ as the estimator in \eqref{eq:eta-stab}, we have
\begin{equation}
\label{eq:efficiency-stab}
\widehat{\eta}_{\mathrm{stab},K}^2
\lesssim \mathrm{osc}(f; K)^2  + \eta_{\mathrm{elem},K}^2  +  \eta_{\mathrm{edge},K}^2 ,
\end{equation}
The constant depends on $k$ and the number of edges on $\partial K$.
\end{theorem}

\begin{proof}
This theorem follows directly from the norm equivalence Lemma \ref{lem:normeq-V}:
\[
\big\vert{\alpha_K^{-1/2}(\operatorname{I}-{\Pi})(\bm{\sigma}_{\mathcal{T}} + \alpha_K \nabla u_{\mathcal{T}})} \big\vert_{S,K}
\lesssim \big\vert{\alpha_K^{-1/2}(\bm{\sigma}_{\mathcal{T}} + \alpha_K \nabla u_{\mathcal{T}})} \big\vert_{S,K},
\]
while evaluating the DoFs $(\mathfrak{e})$ and $(\mathfrak{i})$ using \eqref{eq:sigma-n} and \eqref{eq:sigma-i} reverts us back to the proof of Theorem \ref{thm:efficiency}.
\end{proof}

\begin{theorem}
\label{thm:efficiency-1}
Under the same setting with Theorem \ref{thm:efficiency}, on any $K\in \mathcal{T}_{\mathrm{poly}}$ with 
$\omega_K$ defined as the collection of elements in $\mathcal{T}$ which share at least 1 vertex with $K$
\begin{equation}
\label{eq:efficiency-1}
\widehat{\eta}_{K}
\lesssim \mathrm{osc}(f;K) + \big\Vert \alpha^{1/2}\nabla (u-u_{\mathcal{T}})\big\Vert_{\omega_K},
\end{equation}
with a constant independent of $\alpha$, but dependent on $k$ and the maximum number of edges in 
$K\in \mathcal{T}_{\mathrm{poly}}$. 
\end{theorem}

\begin{proof}
This is a direct consequence of Theorem \ref{thm:efficiency} and \ref{thm:efficiency-stab} and the fact that the residual-based error indicator is efficient by a common bubble function argument.
\end{proof}

\subsection{Reliability}
\label{sec:reliability}
In this section, we shall prove that the computable error estimator $\widehat{\eta}$ is reliable under two common assumptions in the \emph{a posteriori} error estimation literature. For the convenience of the reader, we rephrase them here using a ``layman'' description, for more detailed and technical definition please refer to the literature cited.

\begin{assumption}[$\mathcal{T}$ is $l$-irregular \cite{Carstensen;Hu:2009hanging}]
\label{asp:irregular}
Any given $\mathcal{T}$ is always refined from a mesh with no hanging nodes by a quadsecting red-refinement.  For any two neighboring elements in $\mathcal{T}$, the difference in their refinement levels is $\leq l$ for a uniformly bounded constant $l$, i.e., for any edge $e\in \mathcal{E}$, it has at most $l$ hanging nodes.
 \end{assumption}

By Assumption \ref{asp:irregular}, we denote the father $1$-irregular mesh of $\mathcal{T}$ as 
$\mathcal{T}_1$. On $\mathcal{T}_1$, a subset of all nodes is denoted by $\mathcal{N}_{1}$, which includes the regular nodes $\mathcal{N}_R$ on $\mathcal{T}_1$, as well as $\mathcal{N}_E$ as the set of end points of edges with a hanging node as the midpoint. By \cite[Theorem 2.1]{Carstensen;Hu:2009hanging}, there exists a set of bilinear nodal bases $\{\phi_z\}$ associated with $z\in \mathcal{N}_{1}$, such that $\{\phi_z\}$ form a partition of unity and can be used to construct a 
Cl\'{e}ment-type quasi-interpolation. Furthermore, the following assumption assures that the Cl\'{e}ment-type quasi-interpolant is robust with respect to the coefficient distribution on a vertex patch, when taking nodal DoFs as a weighted average.

\begin{assumption}[Quasi-monotonicity of $\alpha$ \cite{Bernardi;Verfurth:2000Adaptive}]
\label{asp:quasi-monotone}
On $\mathcal{T}$, let $\phi_z$ be the bilinear nodal basis associated with $z\in 
\mathcal{N}_{1}$, with $\omega_{z} :=\operatorname{supp} \phi_z$. For every element $K \subset 
\omega_{z}, K\in \mathcal{T}$, there exists a simply connected element path leading to 
$\omega_{m(z)}$, which is a Lipschitz domain containing the elements where the piecewise constant coefficient 
$\alpha$ achieves the maximum (or minimum) on $\omega_{z}$. 
\end{assumption}

Denote 
\begin{equation}
\label{eq:interp-proj}
\pi_{z} v=
\left\{\begin{array}{ll}
\displaystyle \frac{\int_{\omega_{z} \cap \omega_{m(z)}} v \phi_z }{\int_{\omega_{z} \cap \omega_{m(z)}} 
\phi_z }& \text { if } \bm{z} \in \Omega,
\\ 
0 & \text { if } \bm{z} \in \partial \Omega.
\end{array}\right.
\end{equation}
We note that if $\alpha$ is a constant on $\omega_z$, $ (1, \left(v-\pi_{z} v\right) \phi_{z})_{\omega_z} = 0$. A quasi-interpolation $\mathcal{I}: L^2(\Omega) \to \mathcal{Q}_1(\mathcal{T}_1)$ can be defined as
\begin{equation}
\label{eq:interp}
\mathcal{I} v := \sum_{z\in \mathcal{N}_1} (\pi_z v)\phi_z.  
\end{equation}


\begin{lemma}[Estimates for $\pi_z$ and $\mathcal{I}$]
\label{lem:interp}
Under Assumption \ref{asp:irregular} and \ref{asp:quasi-monotone}, the following estimates hold for any $v\in H^1(\omega_K)$
\begin{equation}
\label{eq:est-interp}
\alpha_K^{1/2} h_K^{-1} \norm{v - \mathcal{I}v}_{K}
+ \alpha_K^{1/2} \norm{\nabla \mathcal{I}v}_{K} \lesssim \big\Vert \alpha^{1/2} \nabla v\big\Vert_{\omega_K},
\end{equation}
and for $\bm{z}\in \mathcal{N}_1$
\begin{equation}
\label{eq:est-proj}
\sum_{K\subset\omega_z} h_{z}^{-2} \|\alpha^{1/2}(v-\pi_{z} v)\phi_z\|_{K}^2 \lesssim 
\big\Vert \alpha^{1/2} \nabla v\big\Vert_{\omega_z}^2,
\end{equation}
in which $h_z := \max_{K\subset\omega_z} h_K$, and here $\omega_K$ denotes the union of elements in $\mathcal{T}_1$ sharing at least a node (hanging or regular) with $K$.
\end{lemma}
\begin{proof}
The estimate for $\pi_z$ follows from \cite[Lemma 2.8]{Bernardi;Verfurth:2000Adaptive}. 
For $\mathcal{I}$, its error estimates and stability only rely on the partition of unity property of the nodal basis set $\{\phi_z\}$ (see e.g., \cite{Verfurth:1999estimates}), therefore the proof follows the same argument with the ones used on triangulations in \cite[Lemma 2.8]{Bernardi;Verfurth:2000Adaptive}. 
\end{proof}

Denotes the subset of nodes $\{\bm{z}\}\subset \mathcal{N}_1$ (i) on the boundary as $\mathcal{N}_{\partial \Omega}$ and (ii) with the coefficient $\alpha$ on patch $\omega_z$ as $\mathcal{N}_I$.  
For the lowest order case, we need the following oscillation term for $f$
\begin{equation}
\label{eq:eta-osc}
\begin{aligned}
  \mathrm{osc}(f;\mathcal{T})^2 := 
& \sum_{z \in \mathcal{N}_1\cap( \mathcal{N}_{\partial \Omega}  \cup \mathcal{N}_I)} 
h_{z}^{2} \big\|\alpha^{-1/2} f\big\|_{\omega_{z}}^2
\\
+&  \sum_{z \in \mathcal{N}_1 \backslash ( \mathcal{N}_{\partial \Omega}  \cup \mathcal{N}_I)} 
h_{z}^{2} \big\|\alpha^{-1/2} (f - f_z)\big\|_{\omega_{z}}^2,
\end{aligned}
\end{equation}
with $f_z := {\int_{\omega_{z}} v \phi_z }/{\int_{\omega_{z}} \phi_z }$.

\begin{theorem}
\label{thm:reliability}
Let \(u_{\mathcal{T}}\) be the solution to problem \eqref{eq:model-fem}, and $\widehat{\eta}$ be the computable error estimator in \eqref{eq:eta-computable}, under Assumption \ref{asp:quasi-monotone} and \ref{asp:irregular}, we have for $k=1$
\begin{equation}
\label{eq:reliability-1}
\big\Vert{\alpha^{1/2}\nabla (u - u_{\mathcal{T}})}\big\Vert \lesssim \left(\widehat{\eta}^2 
+\mathrm{osc}(f;\mathcal{T})^2 \right)^{1/2}.
\end{equation}
For $k\geq 2$,
\begin{equation}
\big\Vert{\alpha^{1/2}\nabla (u - u_{\mathcal{T}})}\big\Vert  \lesssim \widehat{\eta},
\end{equation}
where the constant depends on $l$ and $k$.
\end{theorem}
\begin{proof}
Let $ \varepsilon := u - u_{\mathcal{T}} \in H^1_0(\Omega)$, and $\mathcal{I}\varepsilon\in \mathcal{Q}_1(\mathcal{T}_1) \subset \mathcal{Q}_1(\mathcal{T})$ be the quasi-interpolant in \eqref{eq:interp} of $\varepsilon$, then by the Galerkin orthogonality, $\alpha\nabla u + \bm{\sigma}_{\mathcal{T}} \in \bm{H}(\mathrm{div})$, the Cauchy-Schwarz inequality, and the interpolation estimates \eqref{eq:est-interp}, we have for $k\geq 2$,
\begin{align*}
&\big\Vert{\alpha^{1/2}\nabla \varepsilon}\big\Vert^2  
= \big(\alpha\nabla (u - u_{\mathcal{T}}),  \nabla(\varepsilon -\mathcal{I}\varepsilon)\big)  
\\
=&  \big(\alpha \nabla u + \bm{\sigma}_{\mathcal{T}}, \nabla(\varepsilon -\mathcal{I}\varepsilon)\big)   
- \big(\alpha \nabla u_{\mathcal{T}} + \bm{\sigma}_{\mathcal{T}},  \nabla(\varepsilon -\mathcal{I}\varepsilon)\big) 
\\
= & \big(f -\nabla\cdot \bm{\sigma}_{\mathcal{T}}, \varepsilon -\mathcal{I}\varepsilon\big)   
- \big(\alpha \nabla u_{\mathcal{T}} + \bm{\sigma}_{\mathcal{T}},  \nabla(\varepsilon -\mathcal{I}\varepsilon)\big)
\\
\leq& \left(\sum_{K\in\mathcal{T}} \alpha_K^{-1}h_K^2\norm{f - \nabla\cdot\bm{\sigma}_{\mathcal{T}}}_{K}^2  \right)^{1/2}
 \left(\sum_{K\in\mathcal{T}} \alpha_K h_K^{-2}\norm{\varepsilon - \mathcal{I}\varepsilon}_{K}^2  \right)^{1/2} 
\\ 
&  \left(\sum_{K\in\mathcal{T}} \alpha_K^{-1}\big\Vert{\alpha \nabla u_{\mathcal{T}} + \bm{\sigma}_{\mathcal{T}}} \big\Vert_{K}^2  \right)^{1/2}
 \left(\sum_{K\in\mathcal{T}} \alpha_K \norm{\nabla(\varepsilon - \mathcal{I}\varepsilon)}_{K}^2  \right)^{1/2}.
\\
& \lesssim  \left(\sum_{K\in\mathcal{T}} (\eta_{\mathrm{res},K}^2+\eta_{\mathrm{flux},K}^2) \right)^{1/2} 
 \left(\sum_{K\in\mathcal{T}} \big\Vert{\alpha^{1/2}\nabla \varepsilon}\big\Vert_{\omega_K} \right)^{1/2}.
\end{align*}
Applying the norm equivalence of $\eta$ to $\widehat{\eta}$ by Lemma \ref{lem:normeq-V}, as well as the fact that the number of elements in $\omega_K$ is uniformly bounded by Assumption \ref{asp:irregular}, yields the desired estimate. 

When $k=1$, the residual term on $K$ can be further split thanks to $\Delta \mathbb{Q}_1(K) = \{0\}$. First we notice that by the fact that $\{\phi_z\}$ form a partition of unity,
\begin{equation}
\label{eq:interp-f}
(f, \varepsilon-\mathcal{I}\varepsilon)=\sum_{z \in \mathcal{N}_1} \sum_{K \subset \omega_{z}}
\big(f,\left(\varepsilon-\pi_{z} \varepsilon\right) \phi_{z}\big)_{K},
\end{equation} 
in which a patch-wise constant $f_z$ (weighted average of $f$) can be further inserted by the definition of $\pi_z$ \eqref{eq:interp-proj} if $\alpha$ is a constant on $\omega_z$. 
Therefore, by the assumption of $\alpha_K$ being a piecewise constant, splitting \eqref{eq:interp-f}, we have
\begin{align*}
& \big(f -\nabla\cdot \bm{\sigma}_{\mathcal{T}}, \varepsilon - \mathcal{I}\varepsilon\big)
= \big(f, \varepsilon - \mathcal{I}\varepsilon\big) - \big(\nabla\cdot (\bm{\sigma}_{\mathcal{T}} + \alpha_K\nabla u_{\mathcal{T}}), \varepsilon - \mathcal{I}\varepsilon\big)
\\
= & \sum_{z \in \mathcal{N}} \sum_{K \subset \omega_{z}}
\big(f,\left(\varepsilon-\pi_{z} \varepsilon\right) \phi_{z}\big)_{K} 
- \big(\nabla\cdot (\bm{\sigma}_{\mathcal{T}} + \alpha_K\nabla u_{\mathcal{T}}), \varepsilon - \mathcal{I}\varepsilon\big)
\\
\leq & 
\left(\mathrm{osc}(f;\mathcal{T})^2 \right)^{1/2}
\left(\sum_{z \in \mathcal{N}_1} \sum_{K\subset\omega_z} h_{z}^{-2} 
\|\alpha^{1/2}(\varepsilon-\pi_{z} \varepsilon )\phi_z\|_{K}^2\right)^{1/2}
\\
& \; + \left(\sum_{K\in\mathcal{T}} \alpha_K^{-1} h_K^{2}
\big\Vert \nabla\cdot (\bm{\sigma}_{\mathcal{T}} + \alpha_K\nabla u_{\mathcal{T}}) \big\Vert_{K}^2  \right)^{1/2}
\left(\sum_{K\in\mathcal{T}} \alpha_K h_K^{-2}\norm{\varepsilon - \mathcal{I}\varepsilon}_{K}^2  \right)^{1/2}.
\end{align*}
Applied  an inverse inequality in Lemma \ref{lem:inverse-V} on $\big\Vert \nabla\cdot (\bm{\sigma}_{\mathcal{T}} + \alpha_K\nabla u_{\mathcal{T}}) \big\Vert_{K}$ and the projection estimate for $\pi_z$ \eqref{eq:est-proj}, the rest follows the same argument with the one used in the $k\geq 2$ case.
\end{proof}

\section{Numerical examples}
\label{sec:numerics}
The numerics is prepared using the bilinear element for common AMR benchmark problems. 
The codes for this paper are publicly available on \url{https://github.com/lyc102/ifem} implemented using $i$FEM \cite{Chen:2008innovative}. The linear algebraic system on an $l$-irregular quadtree is implemented following the conforming prolongation approach  
\cite{Cerveny;Dobrev;Kolev:2019Nonconforming} by 
$\mathbf{P}^{\top}\mathbf{A} \mathbf{P}\mathbf{u} = \mathbf{P}^{\top}\mathbf{f}$, where $\mathbf{A}$ is the locally assembled stiffness matrix for all nodes in $\mathcal{N}$, $\mathbf{u}$ and $\mathbf{f}$ are the solution vector associated with $\mathcal{N}_R$ and load vector associated with $\mathcal{N}$, respectively. $\mathbf{P} = (\mathbf{I}, \mathbf{W})^{\top}: \mathbb{R}^{\dim \mathcal{N}_R}\to \mathbb{R}^{\dim \mathcal{N}}$ is a prolongation operator mapping conforming $H^1$-bilinear finite element function defined on regular nodes to all nodes, the weight matrix $\mathbf{W}$ is assembled locally by a recursive $k$NN query in $\mathcal{N}_H$, while the polygonal mesh data structure embedding is automatically built during constructing $\mathbf{P}$. For details we refer the readers to \url{https://github.com/lyc102/ifem/tree/master/research/polyFEM}.

The adaptive finite element (AFEM) iterative procedure is following the standard
\[
\texttt{SOLVE}\longrightarrow \texttt{ESTIMATE} \longrightarrow \texttt{MARK} \longrightarrow \texttt{REFINE}.
\]
 The linear system is solved by MATLAB \texttt{mldivide}. In \texttt{MARK}, the Dorfler $L^2$-marking is used with the local error indicator $\widehat{\eta}_{K}$ in that the minimum subset $\mathcal{M} \subset \mathcal{T}$ is chosen such that
\[
\sum_{K \in \mathcal{M}} 
\widehat{\eta}^{2}_{K} \geq \theta \sum_{K \in \mathcal{T}} \widehat{\eta}^{2}_{K}, \quad \text { for } \theta \in(0,1).
\]
Throughout all examples, we fix $\theta = 0.3$. $\mathcal{T}$ is refined by a red-refinement by quadsecting the marked element afterwards. For comparison, we compute the standard residual-based local indicator for 
$K\in \mathcal{T}_{\mathrm{poly}}$
\[
\eta_{\text{Residual},K}^2 := \alpha_K^{-1}h_K^2 \big\Vert f + \nabla\cdot(\alpha \nabla u_{\mathcal{T}}) \big\Vert_K^2
+  \frac{1}{2}\sum_{e\subset \partial K} \frac{h_e}{\alpha_K + \alpha_{K_e}} 
\big\Vert \jump{\alpha \nabla u_{\mathcal{T}}\cdot\bm{n}_e}{} \big\Vert_e^2,
\]
Let $\eta_{\text{Residual}}^2 = \sum_{K\in \mathcal{T}} \eta_{\text{Residual},K}^2$. The residual-based estimator $\eta_{\text{Residual}}$ is merely computed for comparison purpose and not used in marking.
The AFEM procedure stops when the relative error reaches a threshold. The effectivity indices for different estimators are compared
\[
\text{effectivity index} := {\eta}/{\big\Vert{\alpha^{1/2}\nabla \varepsilon }\big\Vert}, 
\quad \text{ where }\; \varepsilon:=u - u_{\mathcal{T}}, \; \eta = \eta_{\text{Residual}} \text{ or } \widehat{\eta},
\]
i.e., the closer to 1 the effectivity index is, the more accurate this estimator is to measure the error of interest. We use an order $5$ Gaussian quadrature to compute $\Vert{\alpha^{1/2} \nabla (u - u_{\mathcal{T}}) } \Vert$ elementwisely. 
The orders of convergence for various $\eta$'s and $\Vert{\alpha^{1/2} \nabla (u - u_{\mathcal{T}}) } \Vert$ are computed, for which $r_{\eta}$ and $r_{\text{err}}$ are defined as the slope for the linear fitting of $\ln \eta_n$ and 
$\ln \Vert\alpha^{1/2}\nabla(u - u_{\mathcal{T},n}) \Vert$ in the asymptotic regime, 
\[
\ln \eta_n \sim -r_{\eta} \ln N_n + c_1,\quad\text{and}\quad
\ln \Vert{\alpha^{1/2} \nabla (u - u_{\mathcal{T}}) } \Vert \sim -r_{\text{err}} \ln N_n + c_2,
\]
where the subscript $n$ stands for the number of iteration in the AFEM cycles, 
$N_n:= \# ( \mathcal{N}_R\backslash \mathcal{N}_{\partial \Omega})$. $r_{\eta}$ and $r_{\text{err}}$ are considered optimal when being close to $1/2$.

\subsection{L-shaped domain}
\begin{figure}[htp]
  \centering
  \begin{subfigure}[b]{0.43\linewidth}
    \centering\includegraphics[height=120pt]{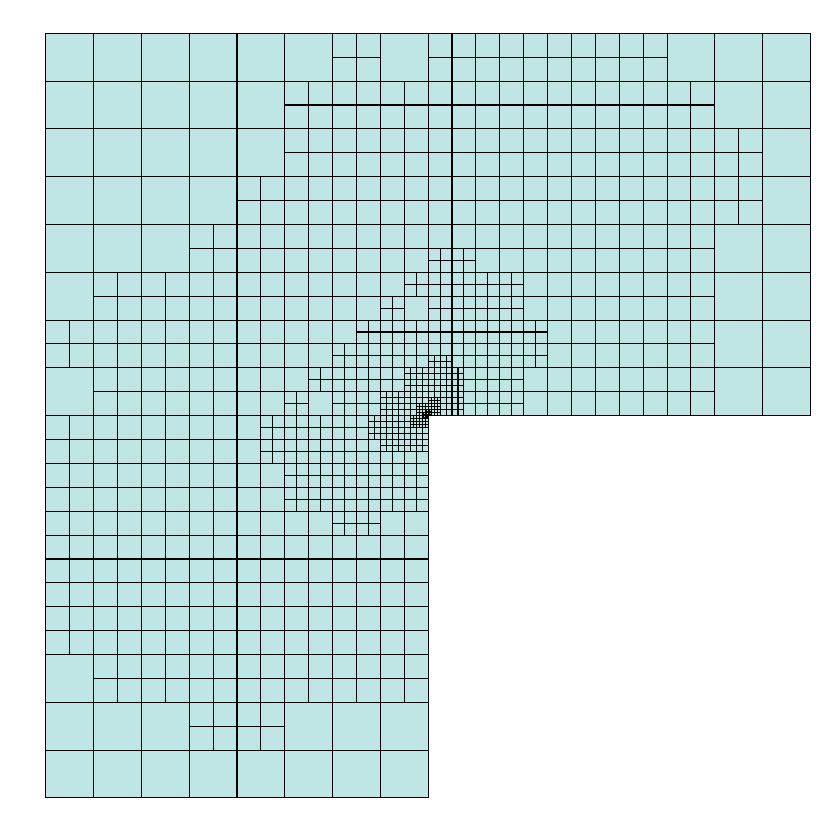}
    \caption{\label{fig:ex1-mesh}}
\end{subfigure}%
\quad
\begin{subfigure}[b]{0.5\linewidth}
      \centering
      \includegraphics[height=120pt,trim=0 0.4cm 0 0]{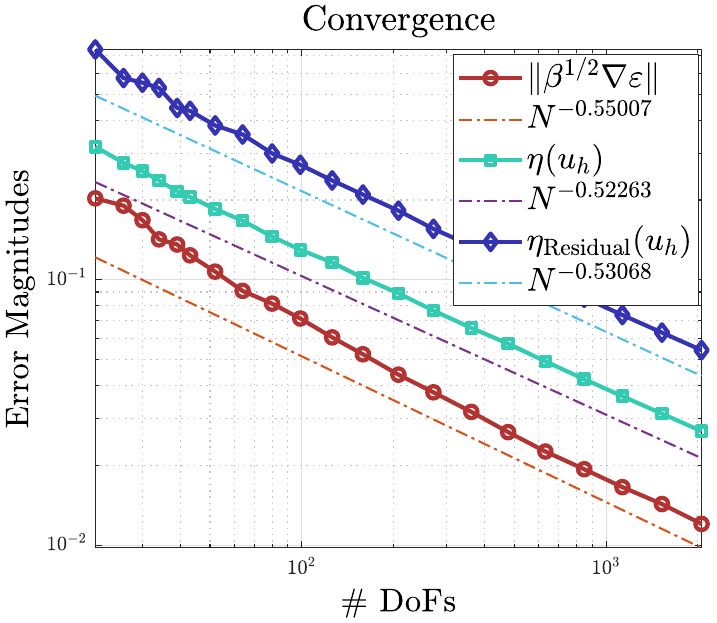}
      \caption{\label{fig:ex1-conv}}
\end{subfigure}
\caption{The result of the L-shape example. (\subref{fig:ex1-mesh}) The adaptively refined mesh with 1014 DoFs.
 (\subref{fig:ex1-conv}) Convergence in Example 1.}
\label{fig:ex1}
\end{figure}

In this example, a standard AMR benchmark on the L-shaped domain is tested. The true solution $u = r^{2/3}\sin(2\theta/3)$ in polar coordinates on $\Omega = (-1,1) \times(-1,1) \backslash[0,1) \times(-1,0]$. The AFEM procedure stops if the relative error has reached $0.01$. The adaptively refined mesh can be found in Figure \ref{fig:ex1-mesh}. While both estimators show optimal rate of convergence in Figure \ref{fig:ex1-conv}, the effectivity index for $\eta_{\mathrm{Residual}}$ is $4.52$, and is 
$2.24$ for $\widehat{\eta}$.

\subsection{A circular wave front}
The solution $u = \tan^{-1}(\alpha(r-r_0))$ is defined on $\Omega = (0,1)^2$ with $r:=\sqrt{(x+0.05)^2+(y+0.05)^2}$, $\alpha = 100$, and $r_0 = 0.7$. The true solution shows a sharp transition layer (Figure \ref{fig:ex2-soln}). The result of the convergence can be found in Figure \ref{fig:ex2-conv}. In this example, the AFEM procedure stops if the relative error has reached $0.05$. Additionally, we note that by allowing $l$-irregular ($l\geq 2$), the AMR procedure shows to be more efficient toward capturing the singularity of the solution. A simple comparison can be found in Figure \ref{fig:ex2-mesh}. The effectivity indices for $\eta_{\mathrm{Residual}}$ and $\widehat{\eta}$ are $5.49$ and $2.08$, respectively.

\begin{figure}[htp]
  \centering
\hspace*{-1cm}   
\begin{subfigure}[b]{0.5\linewidth}
  \includegraphics[width=0.99\linewidth, trim=0 0.8cm 0 0]{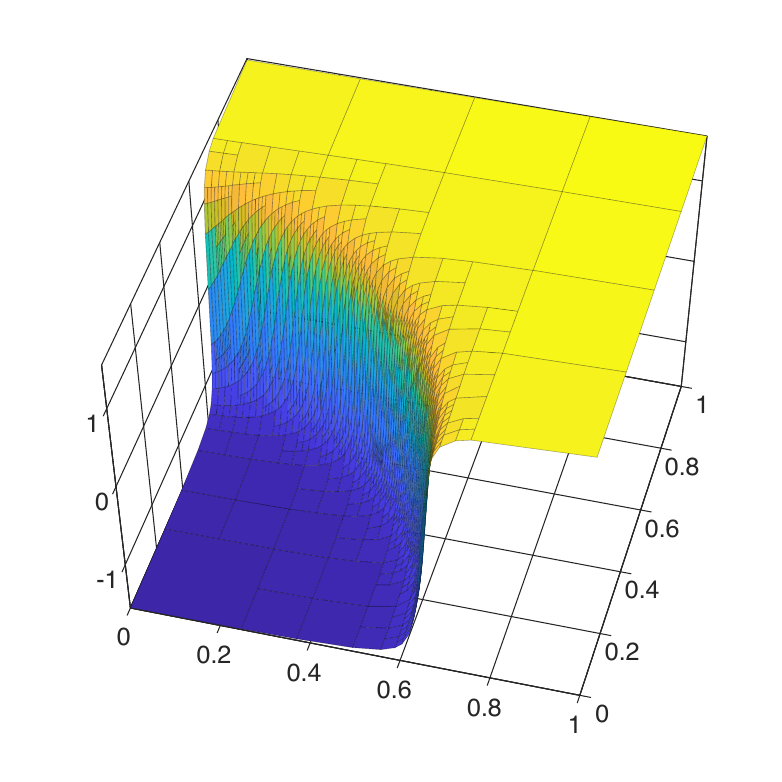}
    \caption{\label{fig:ex2-soln}}
\end{subfigure}%
\;
\begin{subfigure}[b]{0.45\linewidth}
      \centering
      \includegraphics[height=140pt]{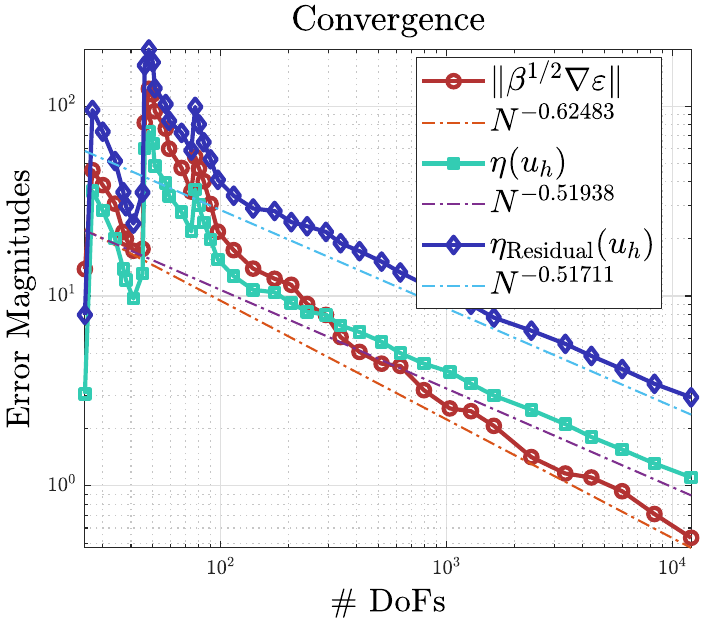}
      \caption{\label{fig:ex2-conv}}
\end{subfigure}
\caption{The result of the circular wave front example. (\subref{fig:ex2-soln}) $u_{\mathcal{T}}$ on a 3-irregular mesh with $\# \mathrm{DoFs} = 1996$, the relative error is $14.3\%$.
 (\subref{fig:ex2-conv}) Convergence in Example 2.}
\label{fig:ex2}
\end{figure}

\begin{figure}[htp]
\centering
\begin{subfigure}[b]{0.48\linewidth}
  \centering\includegraphics[width=0.99\linewidth]{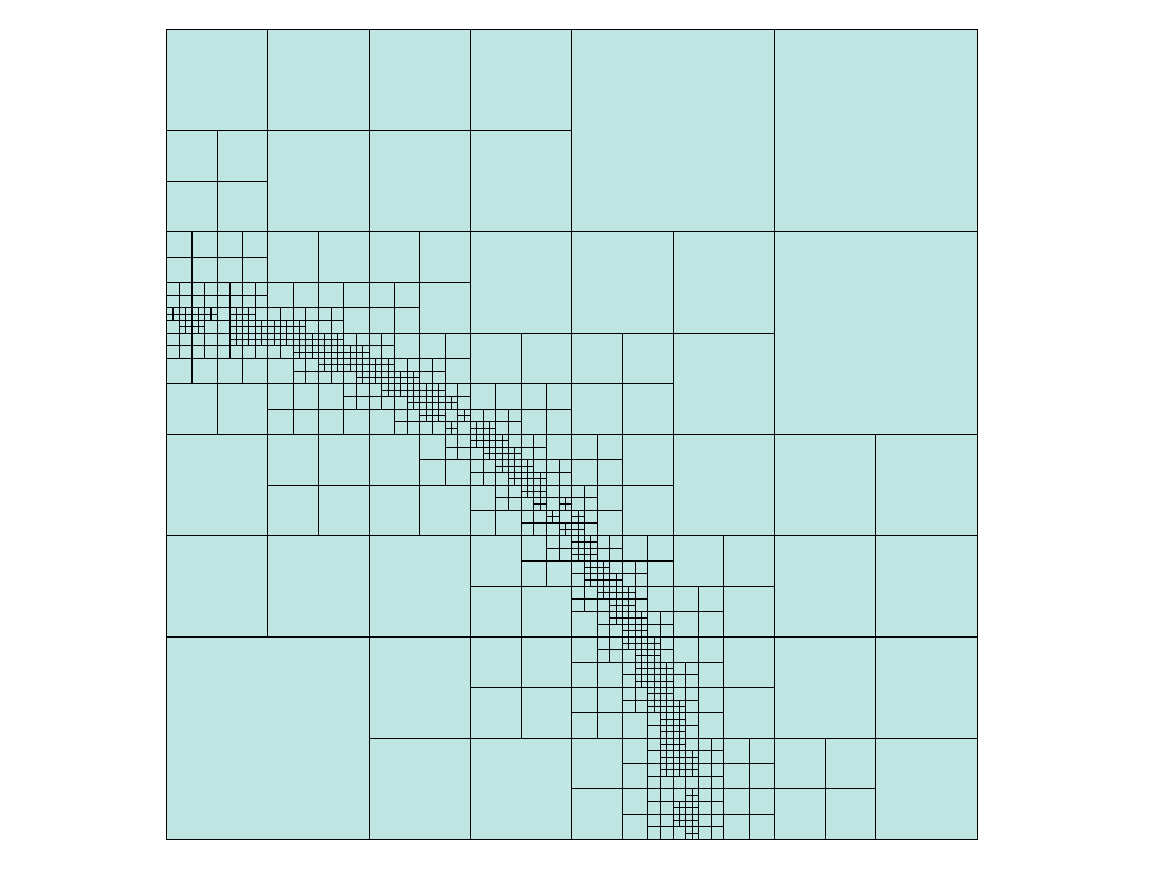}
  \caption{\label{fig:ex2-mesh1}}
\end{subfigure}%
\;
\begin{subfigure}[b]{0.48\linewidth}
      \centering
      \includegraphics[width=0.99\linewidth]{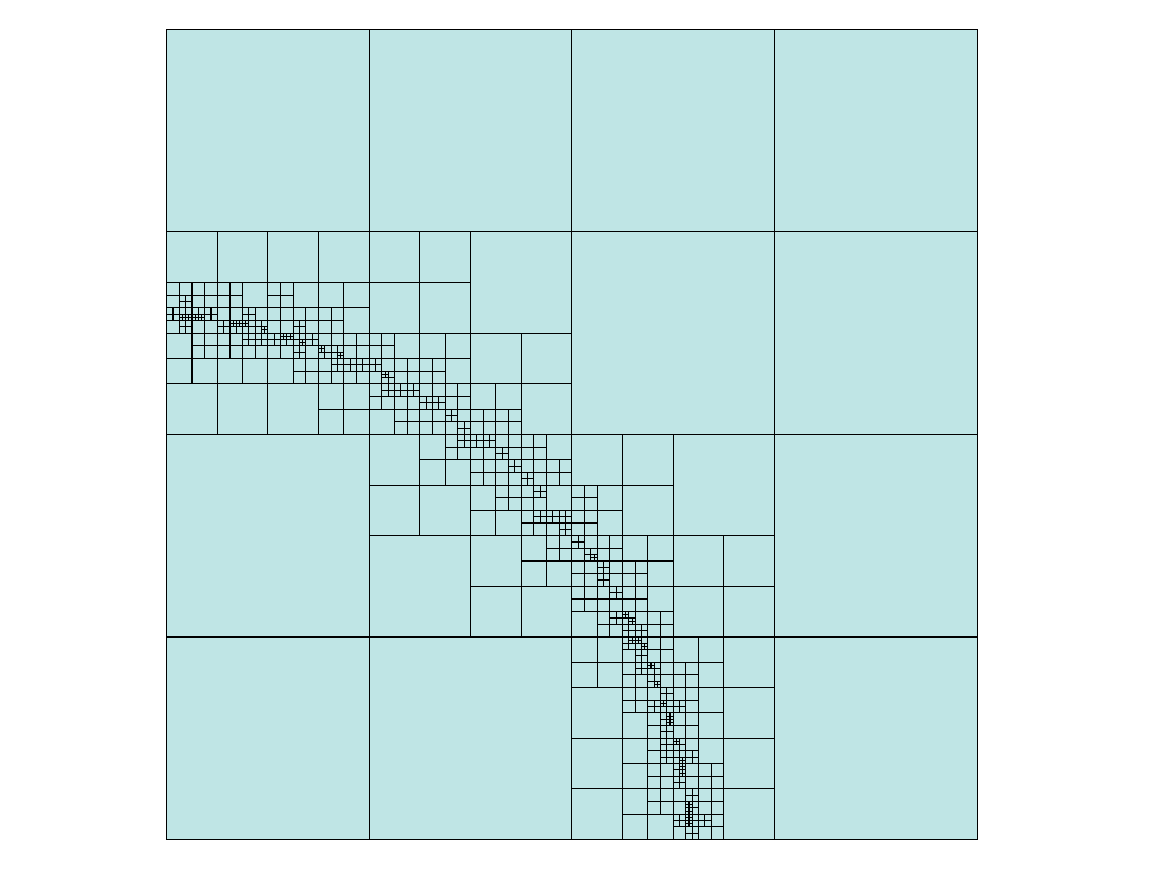}
      \caption{\label{fig:ex2-mesh3}}
\end{subfigure}
\caption{Comparison of the adaptively refined meshes. (\subref{fig:ex2-mesh1}) 1-irregular mesh, $\# \mathrm{DoFs} = 1083$, the relative error is $21.8\%$. (\subref{fig:ex2-mesh3}) 4-irregular mesh, and $\# \mathrm{DoFs} = 1000$, the relative error is $17.8\%$.}
\label{fig:ex2-mesh}
\end{figure}

\subsection{Kellogg benchmark}
This example is a common benchmark test problem introduced in \cite{Bruce-Kellogg:1974Poisson}, see also \cite{Chen;Dai:2002efficiency,Cangiani;Georgoulis;Pryer;Sutton:2017posteriori}) for elliptic interface problems. The true solution $u=r^{\gamma}\mu(\theta)$ is harmonic in four quadrants, and $\mu(\theta)$ takes different values within four quadrants:
\[
\mu(\theta)=\left\{\begin{array}{ll}
\cos ((\pi / 2-\delta) \gamma) \cdot \cos ((\theta-\pi / 2+\rho) \gamma) & \text { if } 0 \leq \theta \leq \pi / 2 
\\ 
\cos (\rho \gamma) \cdot \cos ((\theta-\pi+\delta) \gamma) & \text { if } \pi / 2 \leq \theta \leq \pi 
\\ \cos (\delta \gamma) \cdot \cos ((\theta-\pi-\rho) \gamma) & \text { if } \pi \leq \theta<3 \pi / 2 
\\ \cos ((\pi / 2-\rho) \gamma) \cdot \cos ((\theta-3 \pi / 2-\delta) \gamma) 
& \text { if } 3 \pi / 2 \leq \theta \leq 2 \pi
\end{array}\right.
\]
While $\alpha = R$ in the first and third quadrants, and $\alpha = 1$ in the second and fourth quadrants, and the true flux $\alpha \nabla u$ is glued together using $\bm{H}(\mathrm{div})$-continuity conditions. We choose the folowing set of coefficients for $u$
\[
\gamma = 0.1,\;\; R  \approx 161.4476387975881, \;\; \rho=\pi / 4,\;\; \delta \approx -14.92256510455152,
\]
By this choice, this function is very singular near the origin as the maximum regularity it has is $H^{1+\gamma}_{loc}(\Omega\backslash\{\bm{0}\})$. Through an integration by parts, it can be computed accurately that $\|\alpha^{1/2}\nabla u\| \approx 0.56501154$. For detailed formula and more possible choices of the parameters above, we refer the reader to \cite{Chen;Dai:2002efficiency}. 

The AFEM procedure for this problem stops when the relative error reaches $0.05$, and the resulting mesh and finite element approximation during the refinement can be found in Figure \ref{fig:ex3}, and the AFEM procedure shows optimal rate of convergence in Figure \ref{fig:ex3-conv}. The effectivity index for $\eta_{\mathrm{Residual}}$ is $2.95$, and $1.33$ for $\widehat{\eta}$.

\begin{figure}[h]
  \centering
  \begin{subfigure}[b]{0.4\linewidth}
    \centering\includegraphics[height=120pt]{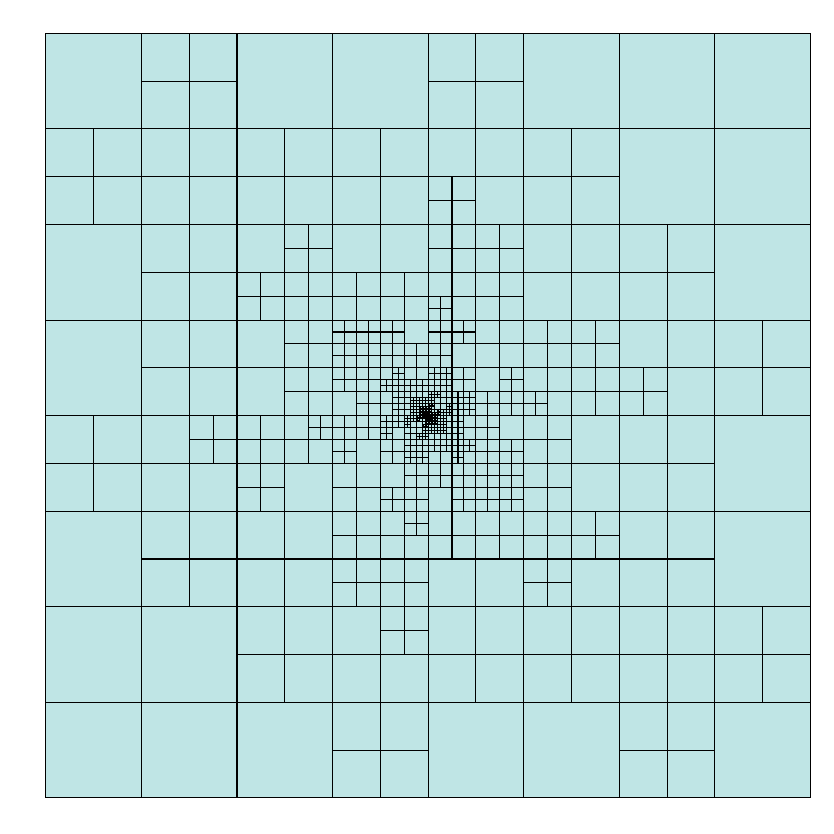}
    \caption{\label{fig:ex3-mesh}}
\end{subfigure}%
\;
\begin{subfigure}[b]{0.58\linewidth}
      \centering
      \includegraphics[height=150pt, trim=0 1.2cm 0 0]{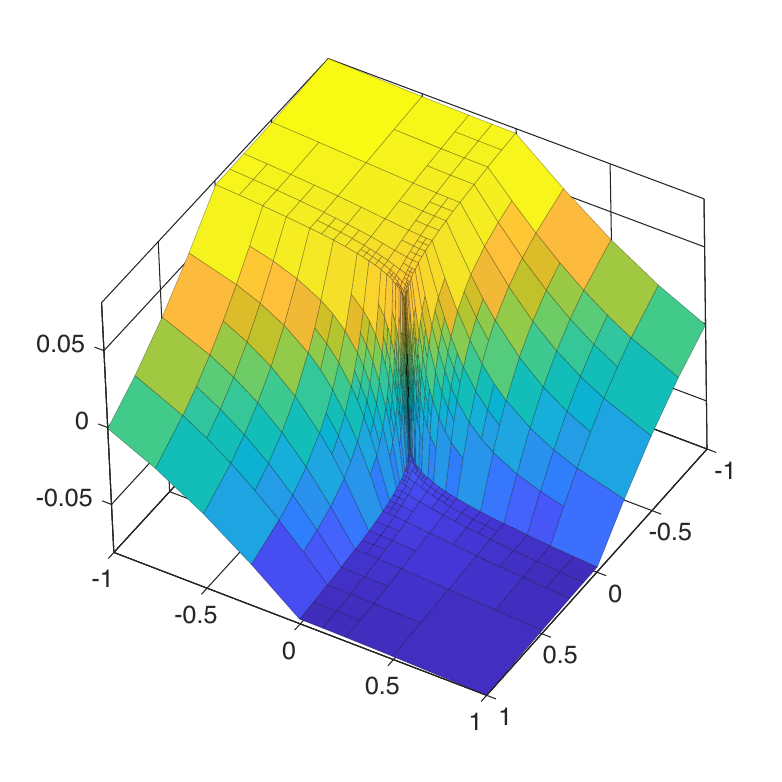}
      \caption{\label{fig:ex3-soln}}
\end{subfigure}
\caption{The result of the Kellogg example. (\subref{fig:ex3-mesh}) The adaptively refined mesh with $\# \mathrm{DoFs} = 2001$ on which the energy error is $0.0753$, this number is roughly $75\%$ of the number of DoFs needed to achieve the same accuracy if using conforming linear finite element on triangular grid (see \cite[Section 4]{Chen;Dai:2002efficiency}). (\subref{fig:ex3-soln}) The finite element approximation with $\# \mathrm{DoFs} = 1736$.}
\label{fig:ex3}
\end{figure}

\begin{figure}[h]
\centering
\includegraphics[height=140pt]{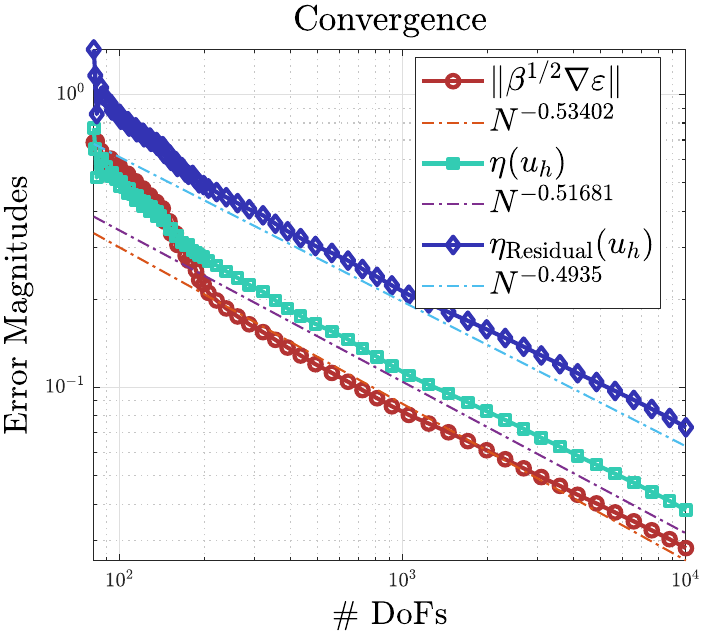}
\caption{The convergence result of the Kellogg example. }
\label{fig:ex3-conv}
\end{figure}

\section{Conclusion}

A postprocessed flux with the minimum $\bm{H}(\mathrm{div})$ continuity requirement is constructed for tensor-product type finite element. The implementation can be easily ported to finite element on quadtree to make use the vast existing finite element libraries in the engineering community. Theoretically, the local error indicator is efficient, and the global estimator is shown to be reliable under the assumptions that (i) the mesh has bounded irregularities, and (ii) the diffusion coefficient is a quasi-monotone piecewise constant. Numerically, we have observed that both the local error indicator and the global estimator are efficient and reliable (in the asymptotic regime), respectively. Moreover, the recovery-based estimator is more accurate than the residual-based one.

However, we do acknowledge that the technical tool involving interpolation is essentially limited to $1$-irregular meshes in reliability. A simple weighted averaging has restrictions and is hard to generalize to $hp$-finite elements, or discretization on curved edges/isoparametric elements. Nevertheless, we have shown that the flexibility of the virtual element framework allows further modification of the space in which we perform the flux recovery to cater the needs.

\section*{Acknowledgments} 
The author is grateful for the constructive advice from the anonymous reviewers. 

\section{Appendix}

\subsection{Inverse estimates and the norm equivalence of a virtual element function}
\label{sec:appendix-norm}
Unlike the identity matrix stabilization commonly used in most of the VEM literature, for $\bm{\tau}\in \mathcal{V}_k(K)$, we opt for a mass matrix/DoF hybrid stabilizer approach. Let $\big\Vert{ \alpha^{-1/2}\bm{\tau}}\big\Vert_{h,K}^2 := (\!(\bm{\tau}, {\bm{\tau}})\!)_{K}$ and
\begin{equation}
\label{eq:innerprod-VRT}
(\!(\bm{\sigma}, {\bm{\tau}})\!)_{K} := \big({\Pi} 
\bm{\sigma}, {\Pi} \bm{\tau} \big)_K
+  {S}_K\big(({\rm I}-{\Pi} )\bm{\sigma}, ({\rm I}-{\Pi} )\bm{\tau}\big),
\end{equation}
where $S_{K}(\cdot,\cdot)$ is defined in \eqref{eq:norm-stab}.

To show the inverse inequality and the norm equivalence used in the reliability bound, on each element, we need to introduce some geometric measures. Consider a polygonal element $K$ and an edge $e\subset \partial K$, let the height $l_e$ which measures how far from this edge $e$ one can advance to an interior subset of $K$, and denote $T_e\subset K$ as a right triangle with height $l_e$ and base as edge $e$.

\begin{proposition}
\label{prop:geometry}
Under Assumption \ref{asp:irregular}, $\mathcal{T}_{poly}$ satisfies 
(1) The number of edges in every $K\in \mathcal{T}_{poly}$ is uniformly 
bounded above. (2) For any edge $e$ on every $K$, $l_e/h_e$ is uniformly bounded below.  
\end{proposition}

\begin{lemma}[Trace inequality on small edges \cite{Cao;Chen:2019Anisotropic}]
\label{lem:trace}
If Proposition \ref{prop:geometry} holds, for $v \in H^1(K)$ and $K\in \mathcal{T}_{\mathrm{poly}}$ we have
\begin{equation}\label{eq:trace}
h_e^{-1/2}\norm{v}_{e} \lesssim h_K^{-1} \norm{v}_{K} + \norm{\nabla v}_{K}, \quad \text{ on } \;e\subset K.
\end{equation}  
\end{lemma}
\begin{proof}
The proof follows essentially equation (3.9) in \cite[Lemma 3.3]{Cao;Chen:2019Anisotropic} as a standard scaled trace inequality on $e$ toward $T_e$ reads
\[
h_e^{-1/2}\norm{v}_{e} \lesssim h_e^{-1} \norm{v}_{T_e} + \norm{\nabla v}_{T_e} 
\lesssim h_K^{-1} \norm{v}_{K} + \norm{\nabla v}_{K}.
\]
\end{proof}

\begin{lemma}[Inverse inequalities]
\label{lem:inverse-V}
Under Assumption \ref{asp:irregular}, we have the following inverse estimates for $\bm{\tau} \in \mathcal{V}_k(K)$ \eqref{eq:space-vhdiv} on any $K\in \mathcal{T}_{poly}$ with constants depending on $k$ and the number of edges in $K$:
\begin{equation}
\label{eq:inverse-V}
 \|\nabla \cdot \bm{\tau}\|_K \lesssim h_K^{-1} \|\bm{\tau}\|_K, \quad 
\text{ and } \quad \|\nabla \cdot \bm{\tau}\|_K  \lesssim h_K^{-1} S_K\big(\bm{\tau},\bm{\tau}\big)^{1/2}.
\end{equation}
\end{lemma}

\begin{proof}
The first inequality in \eqref{eq:inverse-V} can be shown using a bubble function trick. Choose $b_K$ be a bubble function of $T_{e'}$ where $e'$ is the longest edge on $\partial K$. Denote $p := \nabla \cdot \bm{\tau} \in \mathbb{P}_{k-1}(K)$, we have
\[
\|\nabla \cdot \bm{\tau}\|_K^2 \lesssim (\nabla \cdot \bm{\tau}, p b_K)
= -(\bm{\tau}, \nabla (p b_K)) \leq \norm{\bm{\tau}}_K \norm{ \nabla (p b_K)}_K,
\]
and then $\norm{ \nabla (p b_K)}$ can be estimated as follows
\[
\norm{ \nabla (p b_K)} \leq \norm{ b_K \nabla p }_K + \norm{p\nabla b_K}_K
\leq \norm{ b_K }_{\infty,\Omega} \norm{\nabla p }_K + \norm{p}_K \norm{\nabla b_K}_{\infty,K}.
\]
Consequently, the first inequality in \eqref{eq:inverse-V} follows above by the standard inverse estimate for polynomials $\norm{\nabla p}_K\lesssim h_K^{-1} \norm{p}_K$, and the properties of the bubble function $\norm{b_K}_{\infty, K} = O(1)$, and $\norm{\nabla b_K}_{ \infty, K} = O(h_K^{-1})$.

To prove the second inequality in \eqref{eq:inverse-V}, by integration by parts we have
\begin{equation}
\label{eq:inverse-intbyparts}
\norm{\nabla\cdot\bm{\tau}}^2 = (\nabla\cdot\bm{\tau}, p) = -(\bm{\tau},\nabla p)
+ \sum_{e\subset\partial K} (\bm{\tau}\cdot \bm{n}_e, p).  
\end{equation}
Expand $\nabla\cdot \bm{\tau} = p$ in the monomial basis $p(\bm{x}) = \sum_{\alpha\in \Lambda} p_{\alpha} m_{\alpha}(\bm{x})$, and denote the mass matrix $\mathbf{M}: = \big( (m_{\alpha}, m_{\gamma})_{K} \big)_{\alpha\gamma}$, $\mathbf{p}:= (p_{\alpha})_{\alpha\in \Lambda}$, it is straightforward to see that 
\begin{equation}
\label{eq:inverse-scaling}
 \norm{p}_K^2 = \mathbf{p}^{\top} \mathbf{M} \mathbf{p} 
\geq \mathbf{p}^{\top} \operatorname{diag}(\mathbf{M}) \mathbf{p} \geq \min_j \mathbf{M}_{jj}\norm{\mathbf{p}}_{\ell^2}^2
\simeq h_K^2 \norm{\mathbf{p}}_{\ell^2}^2, 
\end{equation}
since $\int_K (x-x_K)^l (y-y_K)^m\,\mathrm{d} x\mathrm{d} y\geq 0$ for the off-diagonal entries of $\mathbf{M}$ due to $K$ being geometrically a rectangle (with additional vertices). As a result, applying the trace inequality in Lemma \ref{lem:trace} on \eqref{eq:inverse-intbyparts} yields
\[
\begin{aligned}
\norm{\nabla\cdot\bm{\tau}}^2
& \leq 
\left(\sum_{\alpha\in \Lambda} (\bm{\tau}, m_{\alpha})_K^2 \right)^{1/2} 
\left(\sum_{\alpha\in \Lambda} p_{\alpha}^2 \right)^{1/2} 
\\
& \quad + 
\left(\sum_{e\subset \partial K} h_e \norm{\bm{\tau}\cdot \bm{n}_e}_e^2 \right)^{1/2}
\left(\sum_{e\subset \partial K} h_e^{-1} \norm{p}_e^2 \right)^{1/2} 
\\
& \lesssim S_K(\bm{\tau},\bm{\tau})^{1/2} \left(\norm{\mathbf{p}}_{\ell^2} + h_K^{-1} \norm{p}_K + \norm{\nabla p}_K \right).
\end{aligned}
\]
As a result, the second inequality in \eqref{eq:inverse-V} is proved when apply an inverse inequality for $\norm{\nabla p}_K$ and estimate \eqref{eq:inverse-scaling}.
\end{proof}

\begin{remark}
While the proof in Lemma \ref{lem:inverse-V} relies on $K$ being a rectangle, the result holds for a much broader class of polygons by changing the basis of $\mathbb{P}_{k-1}(K)$ from the simple scaled monomials to quasi-orthogonal ones in \cite{Mascotto:2017Ill-conditioning,Berrone;Borio:2017Orthogonal} and apply the isotropic polygon scaling result in \cite{Cao;Chen:2019Anisotropic}. 
\end{remark}

\begin{lemma}[Norm equivalence]
\label{lem:normeq-V}
Under Assumption \ref{asp:irregular}, let ${\Pi}$ be the oblique projection defined in \eqref{eq:sigma-proj}, then the following relations holds for $\bm{\tau} \in \mathcal{V}_k(K)$ \eqref{eq:space-vhdiv} on any $K\in \mathcal{T}_{poly}$:
\begin{equation}
\label{eq:normeq-V}
\gamma_* \Vert{\bm{\tau}}\Vert_K \leq  
\Vert{ \bm{\tau}}\Vert_{h,K} \leq 
\gamma^*\Vert{\bm{\tau}}\Vert_K,
\end{equation}
where both $\gamma_*$ and $\gamma^*$ depends  on $k$ and the number of edges in $K$.
\end{lemma}

\begin{proof}
First we consider the lower bound, by triangle inequality,
\[
\Vert{\bm{\tau}}\Vert_{K}\leq 
\big\Vert{{\Pi}\bm{\tau}}\big\Vert_{K} + 
\big\Vert{(\bm{\tau} - {\Pi}\bm{\tau}) }\big\Vert_{K}.
\]
Since ${\Pi}\bm{\tau} \in \mathcal{V}^{k}(K)$, it suffices to 
establish the following to prove the lower bound in \eqref{eq:normeq-V} 
\begin{equation}
\label{eq:normeq-K}
\Vert{\bm{\tau} }\Vert_{K}^2
\leq S_K\big(\bm{\tau},\bm{\tau}\big), \quad \text{ for } 
\bm{\tau}\in \mathcal{V}_k(K).
\end{equation}
To this end, we consider the weak solution to the following auxiliary boundary value 
problem 
on $K$: 
\begin{equation}
\label{eq:normeq-aux}
  \left\{
\begin{aligned}
\Delta \psi &= \nabla\cdot \bm{\tau}&\text{ in } K,
\\
\frac{\partial \psi}{\partial n}  &= 
\bm{\tau} \cdot\bm{n}_{\partial K}  &\text{ on }\partial K.
\end{aligned}
\right.
\end{equation}

By a standard Helmholtz decomposition result (e.g. Proposition 3.1, Chapter 1\cite{Girault;Raviart:1986Finite}), 
we have $\bm{\tau} -\nabla \psi = \nabla^{\perp} \phi$. Moreover, since on $\partial K$, 
$0 = \nabla^{\perp} \phi \cdot \bm{n} = \nabla \phi \cdot \bm{t} = \partial \phi/\partial s$, we can further choose $\phi\in H^1_0(K)$. As a result, by the assumption that $\nabla\times \bm{\tau} = 0$ for $\bm{\tau}$ in the modified virtual element space \eqref{eq:space-vhdiv}, we can verify that
\[
\norm{\bm{\tau} - \nabla \psi}_K^2 = (\bm{\tau} - \nabla \psi,  \nabla^{\perp} \phi) = 0.
\]
Consequently, we proved essentially the unisolvency of the modified VEM space \eqref{eq:space-vhdiv} and $\bm{\tau} = \nabla \psi$. We further note that $\psi$ in \eqref{eq:normeq-aux} can be chosen in $H^1(K)/\mathbb{R}$
and thus%
\begin{equation}
\label{eq:normeq-intbyparts}
\begin{aligned}
& \big\Vert{\bm{\tau} }\big\Vert_{K}^2 = (\bm{\tau}, \nabla \psi)_K = 
\big(\bm{\tau}, \nabla \psi \big)_K
\\
= & \; -\big(\nabla\cdot\bm{\tau}, \psi \big)_K+
(\bm{\tau}\cdot\bm{n}_{\partial K} ,\psi )_{\partial K}
\\
\leq & \;\|\nabla \cdot \bm{\tau}\|_K \| \psi\|_K + 
\sum_{e\subset \partial K} \|\bm{\tau}\cdot\bm{n}_e\|_e\| \psi \|_e
\\
\leq &\; \|\nabla \cdot \bm{\tau}\|_K \| \psi\|_K + \left(\sum_{e\subset \partial K} 
 h_e\|\bm{\tau}\cdot\bm{n}_e\|_e^2\right)^{1/2}
\left(\sum_{e\subset \partial K} h_e^{-1}\|\psi\|_e^2\right)^{1/2}
\end{aligned}
\end{equation}
Proposition \ref{prop:geometry} allows us to apply an isotropic trace inequality 
on an edge of a polygon (Lemma \ref{lem:trace}), combining with the Poincar\'{e} inequality for $H^1(K)/\mathbb{R}$,  we have, on every $e\subset \partial K$,
\[
h_e^{-1/2}\|\psi\|_e \lesssim h_K^{-1} \|\psi\|_K + \|\nabla \psi\|_K
\lesssim \|\nabla \psi\|_K.          
\] 
Furthermore applying the inverse estimate in Lemma \ref{lem:inverse-V} on the bulk term above, we have
\[
\big\Vert{\bm{\tau} }\big\Vert_{K}^2 \lesssim S_K\big(\bm{\tau},\bm{\tau}\big)^{1/2} \|\nabla \psi\|_K,
\]
which proves the validity of \eqref{eq:normeq-K}, thus yield the lower bound.

To prove the upper bound, by $\big\Vert{{\Pi}\bm{\tau}}\big\Vert_{K}\leq 
\Vert{\bm{\tau}}\Vert_{K}$, it suffices to establish the reversed direction of \eqref{eq:normeq-K} on a single edge $e$ and for a single monomial basis $m_{\alpha}\in \mathbb{P}_{k-1}(K)$: 
\begin{equation}
\label{eq:normeq-upper}
h_e\|\bm{\tau}\cdot\bm{n}_e\|_e^2 \lesssim \norm{\bm{\tau}}_K,\quad \text{ and }
\quad |(\bm{\tau}, \nabla m_{\alpha})_K| \leq \norm{\bm{\tau}}_K.
\end{equation}
To prove the first inequality, by Proposition \ref{prop:geometry} again, consider 
the edge bubble function $b_e$ such 
that $\operatorname{supp} b_e = T_e$. We can let $b_e = 0$ on $e'\subset\partial K$ for 
$e'\neq e$. It is easy to verify that:
\begin{equation}
\label{eq:normeq-bubble}
\norm{\nabla b_e}_{\infty,K} = O(1/h_e), \text{ and }
\norm{b_e}_{\infty, K} = O(1).
\end{equation}

Denote $q_e:= \bm{\tau}\cdot\bm{n}_e$, and extend it to $\oversetc{\circ}{K}$ by a constant extension in the normal direction rectangular strip $R_e \subset K$ with respect to $e$ (notice 
$\operatorname{supp} b_e \subset R_e$), we have
\[
\begin{aligned}
\|\bm{\tau}\cdot\bm{n}_e\|_e^2 & \lesssim 
\big(\bm{\tau}\cdot\bm{n}_e, b_e q_e \big)_e = 
x\big(\bm{\tau}\cdot\bm{n}_e, b_e q_e \big)_{\partial K}  
\\
&= \big(\bm{\tau}, q_e\nabla b_e \big)_K
+ \big(\nabla\cdot\bm{\tau}, b_e q_e\big)_K
\\
& \leq \norm{\bm{\tau}}_K  \norm{q_e\nabla b_e}_{T_e}
+ \norm{\nabla\cdot\bm{\tau}}_K \norm{q_e b_e}_{T_e},
\\
& \leq \norm{\bm{\tau}}_K  \norm{q_e}_{T_e} \norm{\nabla b_e}_{\infty,K}
+ \norm{\nabla\cdot\bm{\tau}}_K \norm{q_e}_{T_e} \norm{b_e}_{\infty,K}.
\end{aligned}
\]
Now by the fact that $\norm{q_e}_{T_e} \lesssim h_e^{1/2} \norm{q_e}_e$, the scaling of the edge bubble function in \eqref{eq:normeq-bubble}, and the first inverse estimate of $\norm{\nabla\cdot\bm{\tau}}_K\lesssim h_K^{-1}\norm{\bm{\tau}}_K$ in Lemma \ref{lem:inverse-V} yields the first part of \eqref{eq:normeq-upper}.

The second inequality in \eqref{eq:normeq-upper} can be estimated straightforward by the scaling of the monomials \eqref{eq:monomials-elem}
\begin{equation}
 \left|(\bm{\tau}, \nabla m_{\alpha})_K\right| 
\leq \norm{\bm{\tau}}_K \norm{\nabla m_{\alpha}}_K \leq \norm{\bm{\tau}}_K .
\end{equation}
Hence, \eqref{eq:normeq-V} is proved.
\end{proof}


\end{document}